\documentclass[article]{siamltex}
\usepackage{amstext,amsmath,amsfonts,amssymb,amsbsy,enumerate}
\usepackage{comment}
\usepackage{ae}
\usepackage{algorithm}
\usepackage{algorithmic}
\usepackage{xspace}
\usepackage{graphicx}
\usepackage{psfrag}
\usepackage{color}      
\usepackage{array}
\usepackage{multirow}
\usepackage[normalem]{ulem}
\usepackage{todonotes}

\usepackage{hyperref}

\usepackage{boldfonts,macro}

\newcommand{\bXi}{ \boldsymbol{\mathcal{\chi}} }

\definecolor{darkred}{rgb}{.7,0,0}

\definecolor{green}{rgb}{0,0.7,0}

\definecolor{myblue}{rgb}{0,0,0.7}

\newcommand{\bAg}{{\bf A}_\gamma}
\newcommand{\bAG}{{\bf A}_\Gamma}

\newcounter{assumption}
\newenvironment{assumption}{\refstepcounter{assumption}\equation}{\tag{A\theassumption}\endequation}

\graphicspath{{Figures/}}

\title{A posteriori error estimates for the Laplace-Beltrami operator on parametric $C^2$ surfaces}

\author{
  Andrea Bonito\thanks{Department of Mathematics, Texas A\&M University, College Station TX, 77843; email: {\tt bonito@math.tamu.edu}. 
  Partially supported by NSF Grant DMS-1254618.
}
\and
Alan Demlow\thanks{Department of Mathematics, Texas A\&M University, College Station TX, 77843; email: {\tt demlow@math.tamu.edu}.
Partially supported by NSF Grant DMS-1720369.
}
}

\begin{document}

\maketitle 


\begin{AM} 58J32, 65N15, 65N30
\end{AM}

 \pagestyle{myheadings}
\thispagestyle{plain}
\markboth{A. BONITO AND A. DEMLOW}{A POSTERIORI ESTIMATES ON PARAMETRIC SURFACES}

\begin{abstract}
We prove new a posteriori error estimates for surface finite element methods (SFEM).  Surface FEM approximate solutions to PDE posed on surfaces.  Prototypical examples are elliptic PDE involving the Laplace-Beltrami operator.  Typically the surface is approximated by a polyhedral or higher-order polynomial approximation.  The resulting FEM exhibits both a geometric consistency error due to the surface approximation and a standard Galerkin error.  A posteriori estimates for SFEM require practical access to geometric information about the surface in order to computably bound the geometric error.  It is thus advantageous to allow for maximum flexibility in representing surfaces in practical codes when proving a posteriori error estimates for SFEM.  However, previous a posteriori estimates using general parametric surface representations are suboptimal by one order on $C^2$ surfaces.  Proofs of error estimates optimally reflecting the geometric error instead employ the closest point projection, which is defined using the signed distance function.  Because the closest point projection is often unavailable or inconvenient to use computationally, a posteriori estimates using the signed distance function have notable practical limitations.  We merge these two perspectives by assuming {\it practical} access only to a general parametric representation of the surface, but using the distance function as a {\it theoretical} tool.  This allows us to derive sharper geometric estimators which exhibit improved experimentally observed decay rates when implemented in adaptive surface finite element algorithms.
\end{abstract}

\section{Introduction}

The Laplace-Beltrami operator (or surface Laplacian) has received a great deal of attention recently in part due to its ubiquity in geometric PDEs and in particular in applications involving surfaces that evolve and are the domain of an underlying PDE.  Typical examples are mean curvature flow and surface diffusion appearing in materials science modeling \cite{MR1943134} or Willmore flow as a prototype for equilibrium shapes of membranes governed by bending energy \cite{seifert1997configurations}.

In this paper we consider finite element approximation of solutions to the Laplace-Beltrami problem
\begin{equation}
\label{model_problem}
-\Delta_\gamma u = f \hbox{ on } \gamma.
\end{equation}
Here $\gamma \subset \mathbb{R}^{n+1}$ is an orientable, $C^2$ hypersurface, and $\Delta_\gamma$ is the Laplace-Beltrami operator on $\gamma$.  We mainly focus on the cases where $\gamma$ is closed so that the compatibility condition $\int_\gamma f =0$ must be assumed to guarantee existence of a solution, and we additionally impose $\int_\gamma u =0$ in order to fix a unique solution.  

Dziuk defined a canonical piecewise linear finite element method for approximating solutions to \eqref{model_problem} in \cite{Dziuk:88}.  This method proceeds by approximating $\gamma$ by a polyhedral surface $\Gamma$ having triangular faces which serve as the finite element mesh.  Finite element shape functions are then defined on this mesh and used to approximately solve \eqref{model_problem}.  The corresponding stiffness matrix matches the cotangent formula \cite{wardetzky2008convergence} for the approximation of the surface Laplacian on a polyhedral surface.
This procedure was extended to higher-degree finite element spaces and surface approximations in \cite{Demlow:09}.  A variational crime is committed in this method due to the approximation of $\gamma$ by $\Gamma$, and the resulting consistency error is often called  a ``geometric error''.  If $\Gamma$ is a degree $k$ surface approximation on a quasi-uniform mesh of width $h$ and a degree $r$ finite element space is used to construct the finite element approximation $U$, a priori error analysis yields
\begin{equation}
\label{a_priori}
\|u-U\|_{H^1(\gamma)} \le O(h^r) +O(h^{k+1})
\end{equation}
under the assumption that $\gamma$ and $u$ are sufficiently smooth.  

The assumption that $\gamma$ is $C^2$ is fundamental in this error analysis.  When $\gamma$ is $C^2$ it may be represented {\it implicitly} as the $0$ level set of a signed distance function $d$, and there is also a uniquely defined closest-point projection $\bP_d$ mapping a tubular neighborhood of $\gamma$ onto $\gamma$.   The properties of $\bP_d$ are used integrally in proving the error estimate \eqref{a_priori}.  If on the other hand one were to represent $\gamma$ {\it parametrically} via some {\it arbitrary} smooth map $\bP : \Gamma \rightarrow \gamma$, then the corresponding natural finite element error analysis would only yield
\begin{equation}
\label{bad_a_priori}
\|u-U\|_{H^1(\gamma) } \le O(h^r)+O(h^k).
\end{equation}
The properties of the closest point projection thus lead in effect to a ``geometric superconvergence'' result in which the consistency error is of higher order than one might expect based on generic considerations.  We emphasize that obtaining the superior a priori convergence rate in the geometric error seen in \eqref{a_priori} does not generally required {\it practical computational access} to the closest point projection.  Rather, theoretical use in proofs generally suffices.   

A posteriori error estimates were proved for the case $r=k=1$ in \cite{DemlowDziuk:07} under the assumption that $\gamma$ is $C^2$.  These estimates have the form 
\begin{equation}
\label{apost_form}
\|u-U\|_{H^1(\gamma)} \le \eta + \mathcal{G}_{\bP_d},
\end{equation}
where $\eta$ is a standard residual-type error estimator for controlling energy errors, and $\mathcal{G}_{\bP_d}$ controls the geometric consistency error a posteriori.  Notably, $\mathcal{G}_{\bP_d}$ heuristically retains the ``superconvergent'' a priori order $h^{k+1}=h^2$ that is seen in \eqref{a_priori}.  There are two significant drawbacks to the approach to a posteriori error estimation for \eqref{model_problem} taken in \cite{DemlowDziuk:07}.  First, the assumption that $\gamma$ is $C^2$ may not hold in practice.  Secondly, in contrast to the case of a priori error analysis, the a posteriori estimates of \cite{DemlowDziuk:07} assume {\it practical computational access} to the closest point projection $\bP_d$.  This assumption may be unrealistic.  Closed-form analytical expressions for $\bP_d$ exist only in the very restricted event that $\gamma$ is a sphere or a torus.  If $\gamma$ is computationally represented as the zero level set of {\it some} sufficiently smooth function, then it is possible to approximate $\bP_d$ by for example using a Newton-type algorithm \cite{DemlowDziuk:07, Gr17}.  Practical experience however indicates that this procedure can add significant expense to the code.  Finally, it may be that $\gamma$ is given as a parametric representation.   A posteriori error estimates in which the surface representation follows the framework of \cite{DemlowDziuk:07} have also been proved for discontinuous Galerkin \cite{DM16} and cut \cite{DO12} surface finite element methods along with $L_2$ and $L_\infty$ estimates for Dziuk's method \cite{CD15}.  

An alternate approach to representing $\gamma$ in the context of a posteriori error estimation and adaptivity is given in \cite{BCMMN16, BCMN:Magenes}. In these works $\gamma$ is only assumed to be globally Lipschitz and elementwise $C^{1, \alpha}$.  In addition, $\gamma$ is computationally represented via a parametrization $\bP:\Gamma \rightarrow \gamma$.  The estimators then have the form
\begin{equation}
\label{par_apost_form}
\|u-U\|_{H^1(\gamma)} \le \eta + \mathcal{G}_{\bP},
\end{equation}
where $\eta$ is a standard residual error estimator as above, and $\mathcal{G}_{\bP}$ bounds the geometric consistency error by using information from the parametrization $\bP$.  This framework avoids the two main flaws of the approach of \cite{DemlowDziuk:07}:  It allows for surfaces less regular than $C^2$, and allows for a much more flexible surface representation that may take the form of an implicit representation if it is available, but does not require computational access to the distance map $\bP_d$.  The price that is paid for these advantages is that the geometric consistency estimator $\mathcal{G}_{\bP}$ heuristically only retains the reduced a priori convergence order $O(h^k)$ seen in \eqref{bad_a_priori}, even if $\gamma$ is $C^2$.  In the latter case, adaptive algorithms based on $\mathcal{G}_{\bP}$ generally resolve the geometry much more than is necessary to reach a given error tolerance.  Quasi-optimal error decay for adaptive finite element approximations of $u$ lying in certain regularity classes is derived in \cite{BCMMN16, BCMN:Magenes}.  However, these regularity classes are artificially restricted by the surface approximation when $\gamma$ is $C^2$.  We demonstrate computationally below that over-resolution of the geometry considerably affects the efficiency of the adaptive algorithm.

Our goal in this paper is to produce a posteriori error estimates for finite element approximations to \eqref{model_problem} which combine the major advantages of the parametric and implicit approaches to representing $C^2$ surfaces $\gamma$.  More precisely, we assume that our code has {\it practical} access only to {\it some} reasonable parametric representation $\bP$ of $\gamma$, as in \cite{BCMMN16, BCMN:Magenes,  MMN:11}.  On the other hand, we know in this case that the closest point projection $\bP_d$ exists.  We make {\it theoretical} use of its properties to produce computable a posteriori error estimators that require information only from $\bP$, but which heuristically retain the ``superconvergent'' geometric order $h^{k+1}$ seen in \eqref{a_priori} and $\mathcal{G}_{\bP_d}$.  Our proofs that these estimators are reliable and efficient require a number of sometimes technical steps, but underlying them is the simple observation that the closest point parametrization $\bP_d:\Gamma \rightarrow \gamma$ is optimal in $L_\infty$.  That is, for any other parametrization $\bP:\Gamma \rightarrow \gamma$, 
\begin{equation}\label{e:closest}
|\bx-\bP_d(\bx)| \le  |\bx-\bP(\bx)|, ~~\bx \in \Gamma.
\end{equation}
Note that we do not consider surfaces with less than $C^2$ regularity as many critical properties of the closest point projection do not hold beneath that threshold.   

Finally, we point out that the type of parametric finite element methods considered here are used to approximate time dependent problems such as the mean curvature flow \cite{Dziuk:91}, capillary surfaces \cite{Bansch:01},
surface diffusion \cite{BaMoNo:05,BaGaNu:08},
Willmore flow \cite{BaGaNu:08,MR3508988,MR3712171,BNP:10,Dziuk:07,Rusu:05}, fluid biomembranes \cite{MR3614011,BNP:11}, 
and fluid membranes with orientational order \cite{BDN:10,BDN:12}. 
The analysis of these methods is largely open, but we refer to \cite{BaMoNo:04,MR3508988,MR3614011,DeDz:99,DeDz:00,DeDz:06,DeDzEl:05} as well as the survey \cite{DeDzEl:05} for some of the early work including level set and phase field approaches.

The paper is outlined as follows.  In Section \ref{s:prelim} we define approximations of surfaces and lay out assumptions which must be placed on the resolution of $\gamma$ by its discrete approximations in order for our a posteriori estimates to hold.  In Section \ref{S:Laplace-Beltrami} we prove our a posteriori error estimates.  Section \ref{s:numerics} contains numerical tests illustrating the advantages of our estimates.  Finally Section \ref{s:perspectives} contains some concluding remarks and discussion of possible future research directions.

\section{Preliminaries}\label{s:prelim}

\subsection{Representation of Parametric  Surfaces}\label{S:repres-surface}
%
We assume that the surface $\gamma$ is described as the deformation of a $d$ 
dimensional polyhedral surface $\overline{\Gamma}$ by a globally bi-Lipschitz
\emph{homeomorphism} $\bP:\overline{\Gamma} \rightarrow \gamma \subset \mathbb
R^{n+1}$. 
The overline notation is to emphasize that $\overline{\Gamma}$ is piecewise affine.  Thus there is $\overline{L}>0$ such that for all $\overline{\bx}, \overline{\by} \in \overline{\Gamma}$
\begin{equation}
\label{surf_bi_lipschitz}
\overline{L}^{-1} |\overline{\bx}-\overline{\by}| \le |\tilde{\bx} -\tilde{\by}| \le \overline{L}|\overline{\bx}-\overline{\by}|, \qquad  \hbox{where } \tilde{\bx} = \bP(\overline{\bx}), ~ \tilde{\by}=\bP(\overline{\by}).
\end{equation}

The (closed) facets of $\overline{\Gamma}$ are denoted  $\overline{T}$, and form the collection $\overline{\T}=\{\overline{T} \}$.   We assume that these facets are all simplices.  Extension to other element shapes such as quadrilaterals and to nonconforming discretizations is possible under reasonable assumptions with minor modifications.  
We let $\bP_T:\overline{T} \rightarrow \mathbb R^{n+1}$ be
the restriction of $\bP$ to $\overline{T}$.
This partition of $\overline{\Gamma}$ induces the partition
$\widetilde{\T}=\{\widetilde{T}\}_{\overline{T} \in \overline{\T}}$ of $\gamma$ upon setting
\[
\widetilde{T} := \bP_T(\overline{T}), ~~\overline{T}\in \overline{\T}.
\]
Note that this \emph{non-overlapping} parametrization allows for not necessarily globally $C^2$ parameterizations of $\gamma$.  We additionally define {\it macro patches} 
\[ \overline{\omega}_T=\cup_{\overline{T}', \overline{T}' \cap \overline{T} \neq \emptyset} \overline{T}' \]
and 
\[ \widetilde{\omega}_T =  \bP(\overline{\omega}_T). \]
Finally, we let $h_T= | \overline{T} |^{\frac 1 n}$.  

Let $\widehat{T}$ be the unit reference simplex, which we sometimes refer to as the {\it universal parametric domain}.
We denote by $\overline{\bX}_T:\mathbb R^n\to\mathbb R^{n+1}$ the affine map such that $\overline{T} = \overline{\bX}_T(\widehat{T})$
and we let $\bXi_T := \bP \circ\overline{\bX}_T: \widehat{T} \rightarrow \widetilde{T}$ be the corresponding local parametrization of $\widetilde{T}$. 
We extend this property by assuming that there are patches $\widehat{\omega}_T$, $\overline{T} \in \overline{\T}$, consisting of the universal parametric domain $\widehat{T}$ and other shape-regular simplices of unit size such that $\overline{\bX}_T$ may be extended as a continuous, piecewise-affine bijection such that $\overline{\omega}_T=\overline{\bX}_T(\widehat{\omega}_T)$.   Because $\gamma$ is closed, the domains $\widehat{\omega}_T$ may be constructed so that they are convex, and we assume that this is the case.  The patchwise parametric maps may be easily constructed by sewing together their elementwise counterparts except in a few pathological cases involving very course meshes such as a 6-triangle triangulation of the sphere consisting of two ``stacked'' tetrahedra.  In that case, the neighbors of any element consist of the entire set of elements which cannot be ``flattened out'' (plane and sphere are not homotopic). 

We follow \cite{BP:11} and define the shape regularity constant of the subdivision $\overline{\mathcal T} :=\{ \overline{T} \}_{i=0}^M$ as the smallest constant $\overline c$ such that
\begin{equation}\label{e:shape_reg_init}
\overline c^{-1}{\rm diam}(\overline T) | \bw |  \leq | D \overline{\bX}_T(\hat x) \bw |  \leq \overline c \hspace{2pt} {\rm diam}(\overline T) | \bw|, \qquad \bw \in \mathbb R^n, \qquad \overline T \in \overline{\mathcal T},  
\end{equation}
and assume that $\overline c <\infty$. In the following, we omit to mention the dependency on $\overline c$ of the constants appearing in our argumentation. 
We additionally assume that the number of elements in each patch $\widehat{\omega}_T$ is uniformly bounded.  This assumption automatically follows from shape regularity for triangulations of Euclidean domains, but the situation is more subtle for surface triangulations as illustrated in Figure~\ref{f:valence}.  Such a bound does for example hold if $\overline{\Gamma}$ is derived by systematic refinement of an initial surface mesh with a uniform bound on the number of elements in a patch \cite{DemlowDziuk:07}, or more generally using adaptive refinement strategies \cite{BCMMN16,BCMN:Magenes}.  In addition, this implies that all elements in $\overline{\omega}_T$ have uniformly equivalent diameters, as for shape regular triangulations on Euclidean domains.  Finally, because each vertex has uniformly bounded valence the number of parametric patches $\widehat{\omega}_T$ needed may be less than $\#\T$, that is, $\widehat{\omega}_T=\widehat{\omega}_{T'}$ may hold for $\overline{T} \neq \overline{T}'$.  We thus assume that there are a finite collection of such parametric patches independent of the discretization, and that properties of these patches such as constants in extension operators are uniform across the collection.  
\begin{figure}[ht!]
\centerline{\includegraphics[width=0.3\textwidth]{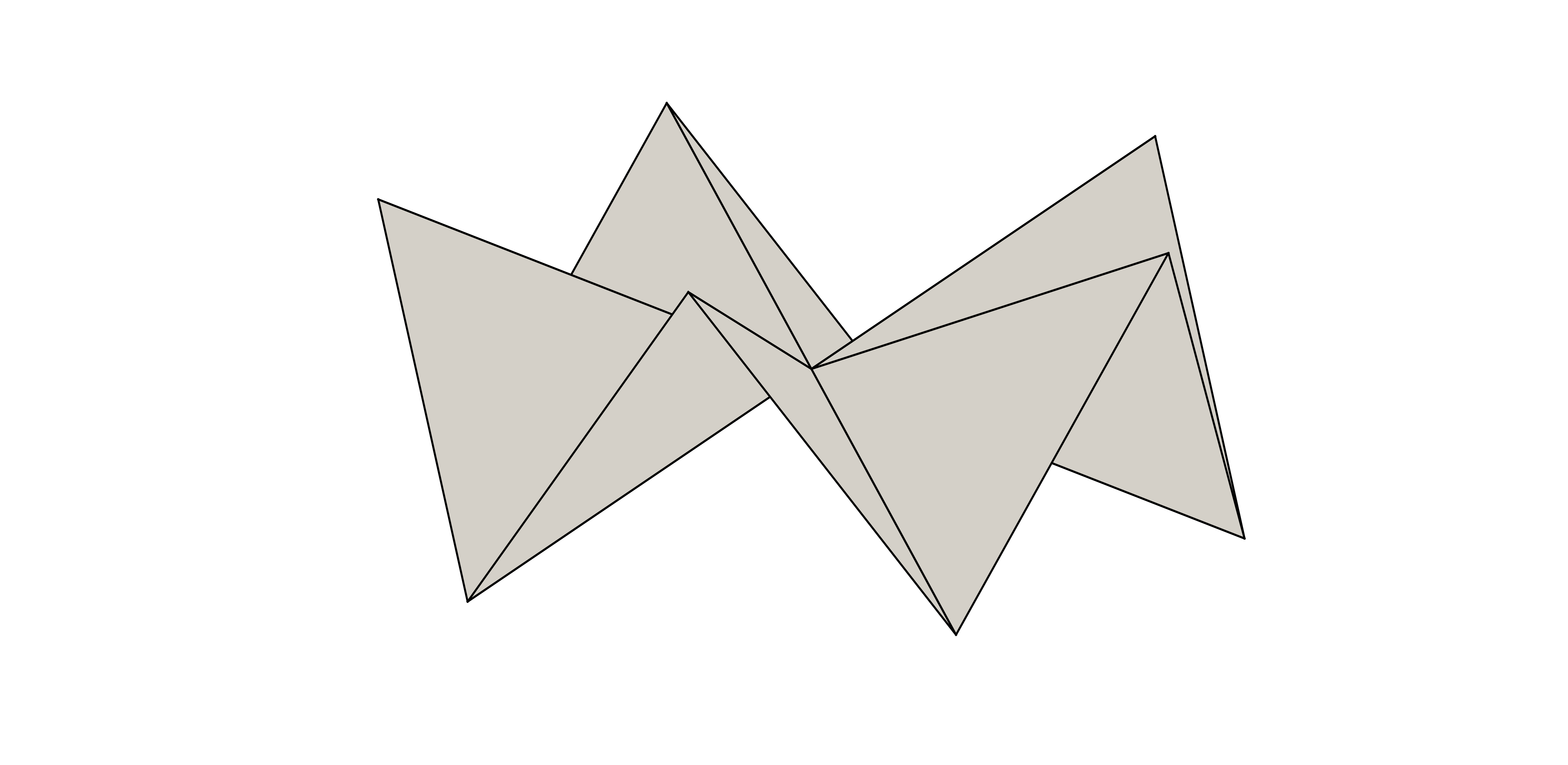}
\includegraphics[width=0.3\textwidth]{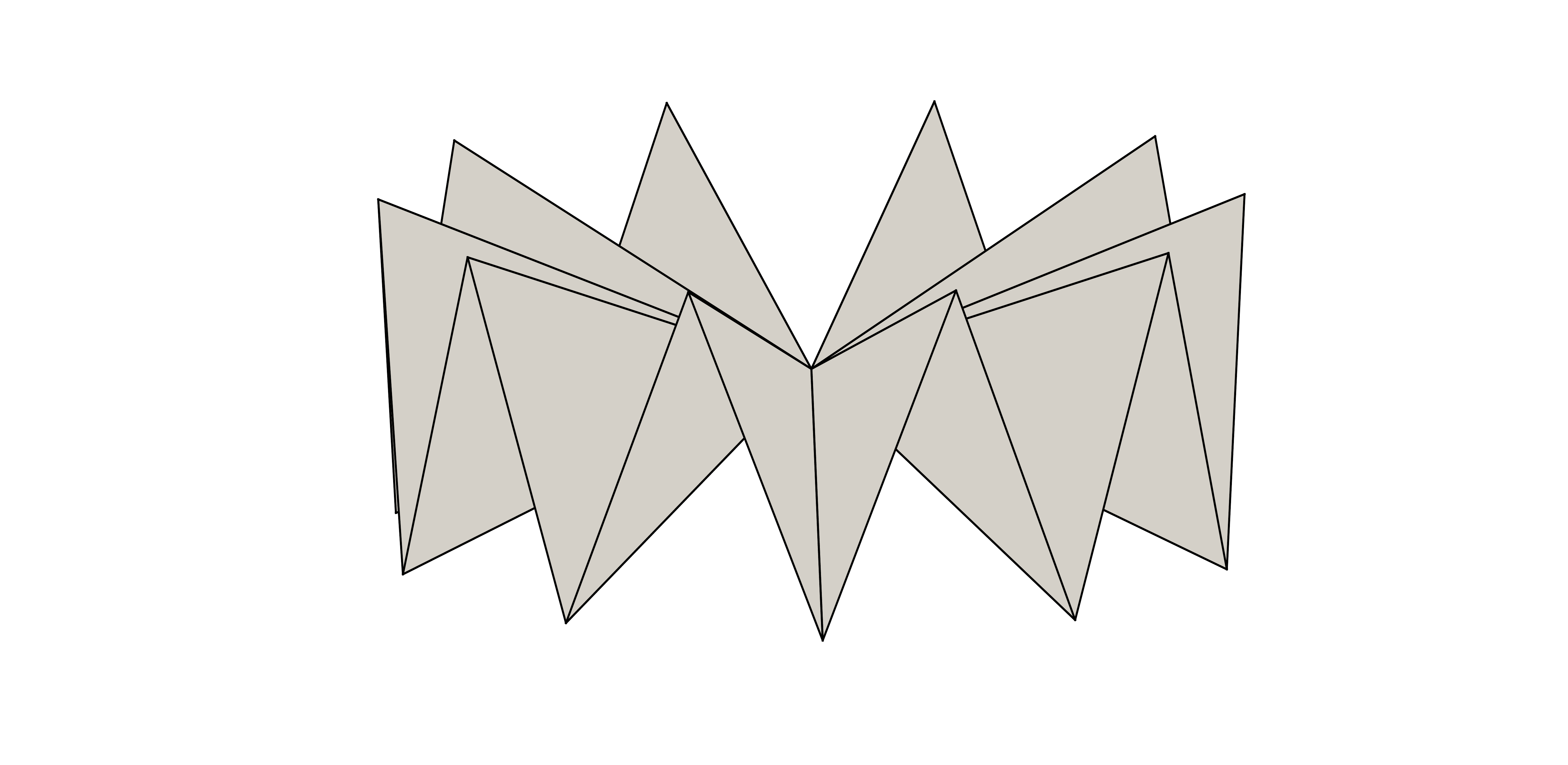}}
\caption{
Two different configurations when $n=2$ illustrating that the number of elements sharing the same vertex could be arbitrarily large even when using triangles satisfying \eqref{e:shape_reg_init} with a uniform constant $\overline{c}$.} \label{f:valence}
\end{figure}


We also extend $\bXi_T = \bP \circ\overline{\bX}_T:\widehat{\omega}_T \to\widetilde{\omega}_T$ to be a local parametrization of $\gamma$ which is bi-Lipschitz with element scaling.  Namely, there exists
a universal constant $L\geq1$ such that for each fixed $\overline{T} \in \overline{\T}$ and for all $\tilde \bx, \tilde \by \in \widetilde{\omega}_T$, 
\begin{equation}\label{bi_lipschitz}
h_T L^{-1} |\hat{\bx}-\hat{\by}| 
\leq |\tilde \bx-  \tilde \by | 
\leq L h_T |\hat{\bx}-\hat{\by}|, 
~~ \textrm{where }  \tilde \bx = \bXi_T(\hat{\bx}), \ \tilde \by = \bXi_T(\hat{\by}). 
\end{equation}
The collection of these parametrizations is denoted $\bXi $, i.e. $\bXi = \{ \bXi_T \}_{\overline{T} \in \overline{\T}}$.
We further assume that $\bP(\mathbf v) = \mathbf v$ 
for all vertices $\mathbf v$ of $\overline{\Gamma}$, so that $\overline{\bX}_T$ is the nodal interpolant of $\bXi_T$ into linears.  

Finally, we note that a function $v_T:\widetilde{T}\to\mathbb R$ defines uniquely two
functions $\hat{v}_T:\widehat{T}\to\mathbb R$ and $\overline{v}_T:\overline{T}\to\mathbb R$ via the
maps $\bXi_T$ and $\bP$, namely
\begin{align}  \label{para_lift}
\hat{v}_T(\hat{\bx}) := v_T(\bXi_T(\hat{\bx})) \quad \forall\; \hat{\bx}\in
\widehat{T} \qquad\text{ and }\qquad \overline{v}_T(\bar{\bx}) := v_T(\bP(\bar{\bx}))
\quad\forall\; \bar{\bx}\in\overline{T}.
\end{align}
Conversely, a function $\hat{v}_T:\widehat{T}\to\mathbb R$ (respectively,
$\overline{v}_T:\overline{T}\to\mathbb R$) defines uniquely the two functions
$v_T:\widetilde{T}\to\mathbb R$ and $\overline{v}_T:\overline{T}\to\mathbb R$ (respectively,
$v_T:\widetilde{T}\to\mathbb R$ and $\hat{v}_T:\widehat{T}\to{\mathbb R}$).
When no confusion is possible, we will always denote by $v$ the two lifts $\bar{v}$
or $\hat{v}$ of $v:\widetilde{T}\to{\mathbb R}$ and set $\tilde{\bx}_T:=\bXi_T(\hat{\bx})$ for all $\hat{\bx}\in\widehat{T}$.

Before proceeding further, we note that as a general rule, we use hat symbols to denote quantities related to $\widehat{T}$, an overline to refer to quantities on $\overline{\Gamma}$, tilde to characterize quantities in $\gamma$ and bold to indicate vector quantities.

\subsection{Interpolation of Parametric Surfaces and Finite Element Spaces}\label{S:interp--surface}
%

\subsubsection{Finite Element Spaces and Surface Approximations}\label{ss:fem}

For $s\geq 1$, let $\mathcal P_s$ be the space of polynomials of degree at most $s$ and let $I_s: C^0(\widehat{T})\to \mathcal P_s$ be the
corresponding Lagrange interpolation operator. Note that we will use the same notation for the componentwise Lagrange interpolation operator
$I_s: C^0(\widehat{T})^{d}\to (\mathcal P_s)^d$, and also for the naturally defined Lagrange interpolant on the faces of $\overline{\Gamma}$.

We fix $k\geq 1$ and define $\bX_T:=I_k \bXi_T$ to be the (component wise) interpolant of degree at most $k$ of
$\bXi_T$.
Then $T:=T_k := \bX_T(\widehat{T})$ is the piecewise polynomial interpolation
of $\widetilde{T}$ with associated subdivision $\T = \{ \bX_T(\widehat{T})\}_{\overline{T} \in \overline{\T}}$.
The global quantities are  defined by
$$
\bX:=\Big\{ \bX_T \Big\}_{\overline{T} \in \overline{\T}}, \qquad \Gamma := \cup_{\overline{T} \in \overline{\T}} T
$$
and we denote by $S_\T$ the set of interior faces of $\T$.  {Note that may equivalently write $T=(I_k \bP)(\overline{T})$, $\overline{T} \in \overline \T$, where $I_k$ is the Lagrange interpolant naturally defined on the faces of $\gamma$.  In order to define a parametric map from $\Gamma$ to $\gamma$, we let
$$\bP:\Gamma \rightarrow \gamma, \qquad \bP(\bx)=\bP \circ  (I_k \bP)^{-1}(\bx), ~~\bx \in \Gamma.$$
Note that use the same notation for $\bP:\overline{\Gamma} \rightarrow \gamma$ and $\bP:\Gamma \rightarrow \gamma$.  The definitions are consistent, and the domain should be clear from the context.  Thus this ambiguity in notation should cause no confusion.  }

For $T\in \T$, we define its diameter as the diameter of the corresponding element in $\overline{\T}$, i.e.   $h_T = |\overline{T}|^{\frac{1}{n}}$ with
$\overline{T} = (\overline{\bX} \circ \bX^{-1})(T)$.
As for $\overline{\mathcal T}$, we assume that there exists $c>0$ such that
\begin{assumption}\label{e:shape_reg}
c^{-1} {\rm diam}(\overline T) | \bw |  \leq | D \bX_T(\hat x) \bw |  \leq c \hspace{2pt} {\rm diam}(\overline T) | \bw|, \qquad \bw \in \mathbb R^n,
\end{assumption}
i.e. the shape regularity constant $\mathcal T$ is strictly positive.
This can be guaranteed from \eqref{e:shape_reg_init} and assuming that the resolution of $\gamma$ by $\Gamma$ is fine enough \cite{BCMMN16, BCMN:Magenes,BP:11}.
However in this work the focus is on deriving a-posteriori error estimators and we assume \eqref{e:shape_reg} directly. 
Moreover, we do not explicitly mention the dependency on $c$ of the constants appearing in our argumentation below.  Finally, we denote by $\omega_T$ the element patch given by $\omega_T=\bX_T(\widehat{\omega}_T)$.  


We fix $r \geq 1$ and define the finite element space on the approximate surface $\Gamma$ by
\begin{equation}\label{d:fem_space}
\V(\T) :=\left\lbrace V \in C^0(\Gamma) \;\big|\; 
	   V|_{T}  = \widehat V \circ \bX^{-1} \mbox{ for some } \widehat V\in \mathcal P_r, ~T \in \T,~ \int_\Gamma V = 0\right\rbrace.
\end{equation}

At this point we emphasize that the integer $k$ denotes the polynomial degree for the approximation of $\gamma$ by $\Gamma$ while $r$ denotes the polynomial degree for the finite element approximation of $u$.
In general, they do not need to be equal.

\subsubsection{Geometric Estimators}\label{ss:geom_estim}
We define in this section an estimator for the \emph{geometric error} due to the approximation of $\gamma$ by a piecewise polynomial surface $\Gamma$. 
For $T \in \T$, we define the \textit{geometric element indicator}
\begin{equation}
\lambda_T:= \| \smash { \nabla (\bP-I_k \bP)}\|_{ L_\infty(\overline{T})},
\end{equation}
and the corresponding \textit{geometric estimator}
\begin{equation}\label{geo-estimator}
\lambda :=\max_{ T \in \T} \lambda_T.
\end{equation}
In \cite{BCMMN16} an estimator equivalent to $\lambda$ is used in order to control the geometric consistency error.  However,  $\lambda$ has a heuristic a priori order of $h^{k}$ for smooth surfaces, and the corresponding a priori error estimate is \eqref{bad_a_priori}.   In view of \eqref{a_priori}, the geometric error should instead be $O(h^{k+1})$ for smooth surfaces, and we realize that the geometry approximation  is overestimated by $\lambda$ when considering $C^2$ surfaces. 
Let 
\begin{equation}
\label{d:beta}
\beta_T:=\|\bP-I_k \bP\|_{L_\infty(\overline{T})}, ~~\beta:=\max_{T \in \T} \beta_T.
\end{equation}
As we shall see, in this context the appropriate estimator is 
\begin{equation}\label{d:mu}
\mu_T := \beta_T + (\lambda_T)^2, \qquad \mu := 
\max_{T \in \T} \mu_T.
\end{equation}
When $\gamma$ is sufficiently smooth, $\mu \lesssim h^{k+1} + h^{2k} \lesssim h^{k+1}$, so estimating the geometry using $\mu$ preserves the ``superconvergent'' geometric consistency error observed in \eqref{a_priori}.

\subsection{Basic differential geometry}

\subsubsection{The distance function map}\label{ss:distfunc}

The structure of the map $\bP$ depends on the application. 
The most popular is when $\widetilde{T}$
is described on $T$ via the \emph{distance function}
$d:\mathcal N \rightarrow \gamma$.  $d$ is well  defined on a local tubular neighborhood $\mathcal N$ given the smoothness assumption on $\gamma$.  Let 
\begin{equation}\label{e:lift_dist}
\widetilde{T} \ni \tilde{\bx} = \bx - d(\bx) \nabla d(\bx) = : \bP_{d}(\bx)  \qquad
\forall~\bx \in \mathcal N.
\end{equation}
Notice that we introduced the notation $\bP_{d}$ for the lift given by the distance function, i.e. the closest point in $\gamma$, as it will appear multiple times in our analysis below. 
We assume from now on that $\overline{\Gamma} \subset \mathcal N$.
Notice that for $\overline{\bx} \in \overline{\Gamma}$, $\bP(\overline{\bx})$ (defined by a generic $\bchi$) and $\bP_{d}(\overline{\bx})$ are both on $\gamma$ but are not necessarily the same points. 
We postpone the discussion of this point until Section~\ref{ss:res1}, see \eqref{e:closest}.

In this work and as in \cite{BCMMN16, BCMN:Magenes,BP:11}, we do not assume that the distance function is available to the user. 
However, the $C^2$ regularity assumed on $\gamma$ will allow us to use the distance function as a theoretical tool to improve upon the geometric estimators provided in \cite{BCMMN16,BCMN:Magenes}. 

We further describe some basic geometric notions.  Given $\bx \in \mathcal{N}$,  $\bnu (\bx): = \nabla d (\bP_{d}(\bx)) = \nabla d (\bx)$ is the normal to $\gamma$ at the closest point $\bP_d(\bx) \in \gamma$, and $\bW(\bx):=D^2 d(\bx)$ is the Weingarten map.  Let also $\kappa_i(\bx)$, $i =1,.., n$, be the eigenvalues of $\bW(\bx)$.  These are the principal curvatures of $\gamma$ if $\bx \in \gamma$ and of parallel surfaces if $\bx \notin \gamma$.    In addition, we have the relationship 
\begin{equation}\label{kappas}
\kappa_i(\bx) = \frac{\kappa_i(\bP_d(\bx))}{1+d \kappa_i(\bP_d(\bx))}.  
\end{equation}
For the sake of convenience we also define the maximum principal curvature 
\begin{equation}
\label{K:def}
K(x)=\max_{1 \le i \le n} |\kappa_i(x)|, ~~x \in \mathcal{N}. 
\end{equation}
We may now more precisely write
\begin{equation}\label{N:def}
\mathcal{N} = \{ \bx \in \mathbb{R}^{n+1}: {\rm dist}(\bx, \gamma) < \| K \|_{L_\infty(\gamma)}^{-1} \}.
\end{equation}

\subsubsection{Differential operators and area elements}\label{S:diff-geom}
%

In this subsection we recall basic differential geometry notations and definitions and refer to \cite{BCMN:Magenes} for details. 

Let $\bG\in {\mathbb R}^{( n +1)\times n}$ be the matrix
$$
\bG := \bG_{\gamma} := [\widehat{\partial}_1\bXi, \ldots, 
\widehat{\partial}_d\bXi],
$$
whose $j$-th column $\widehat{\partial}_j\bXi\in{\mathbb R}^{n+1}$ 
is the vector of partial derivatives of $\bXi$
with respect to the $j^{th}$ coordinate of $\widehat{T}$.
Note that because $\bXi$ is a diffeomorphism, the set $\{\widehat{\partial}_j\bXi\}_{j=1}^d$ consists of
tangent vectors to $\gamma$ which are linearly
independent and forms a basis of the tangent plane of $\gamma$.
The {\it first fundamental
form} of $\gamma$ is the symmetric and positive definite matrix
$\bg\in{\mathbb R}^{n\times n}$ defined by
\begin{equation}\label{1st-form}
\bg= \big( g_{ij} \big)_{1\leq i,j \leq d} := \big(
\widehat{\partial}_i\bXi^T\widehat{\partial}_j\bXi \big)_{1\leq i,j \leq d} = \bG^T\bG.
\end{equation}
We will also need $\widetilde{\bD} =
\widetilde{\bG}^{-1}$, where
\[
\widetilde{\bG} := \big[ \bG, \bnu^T \big] = \big[\widehat{\partial}_1\bXi, \dots,
\widehat{\partial}_d\bXi, \bnu^T\big],
\]
and  $\bD \in {\mathbb R}^{n\times (n+1)}$ resulting from
$\widetilde{\bD}$ by suppressing its last row.
With this notation, we have for $v : \gamma \to {\mathbb R}$
\begin{equation*}
{\nabla}\hat{v} = \nabla_{\gamma} v\bG \qquad  \nabla_{\gamma} {v} = \nabla \hat{v}\; \bD.
\end{equation*}

With the surface area element defined as
\begin{equation}\label{area}
q := \sqrt{\det \bg},
\end{equation}
the Laplace-Beltrami operator on $\gamma$ is given by
\begin{equation}
\label{lb_def}
\Delta_\gamma v = \frac{1}{q} {\text{div}} 
\big(q {\nabla}\hat{v}\bg^{-1}\big).
\end{equation}
We could alternately compute the elementary area on $\gamma$ via the map $\bP_d \circ \bX$ instead of via $\bXi=\bP \circ \bX$, and we denote by $q_d$ the area element obtained by doing so.

The discussion above applies as well to the piecewise polynomial surface
$\Gamma$ (recall that we dropped the index specifying the considered patch). 
The corresponding matrix quantities are indexed by $\Gamma$, so that
\begin{equation}\label{grad-Gamma}
{\nabla}\hat{v} = \nabla_{\Gamma} v\bG_\Gamma, \qquad  \nabla_{\Gamma} {v} = \nabla \hat{v}\; \bD_\Gamma.
\end{equation}
The first fundamental form and the elementary area on $\Gamma$ are denoted $\bg_\Gamma$ and $q_\Gamma$.
Finally, the Laplace-Beltrami operator on $\Gamma$ reads
\begin{equation}\label{lap-bel}
\Delta_\Gamma v = \frac{1}{q_\Gamma} {\text{div}} 
\big(q_\Gamma {\nabla}\hat{v}\bg_\Gamma^{-1}\big).
\end{equation}

In addition, we recall that for $T\in \T$ and $S$ a side of $T$, the outside pointing unit co-normals $\bn$ of $T$ and $\widehat \bn$ of $\widehat{T}$ are related by the following expressions
\begin{equation}\label{eq:hat_n}
\widehat \bn = \bn \bG_\Gamma \frac{r_\Gamma}{q_\Gamma} \qquad \text{and} \qquad  \bn  = \frac{q_\Gamma}{r_\Gamma} \widehat \bn \bD_\Gamma,
\end{equation}
where $r_\Gamma$ is the elementary area associated with the subsimplex $\widehat S:=\bXi^{-1}(S)$.
Hence, the tangential gradient of $v$ in the direction $\bn$ on $S$ is given by 
\begin{equation}\label{e:normal_grad}
\nabla_\Gamma v \cdot \bn = \frac{q_\Gamma}{r_\Gamma} \nabla  v \bg_\Gamma^{-1}|_{\widehat S} \cdot \widehat \bn.
\end{equation}

\subsection{Constants and inequalities}  We use different sets of constants and notation depending on the situation.  We employ a constant $C$ (or $C_j$) which may depend on the space dimension $n$, the polynomial degrees $r$ and $k$, and shape regularity constants, but is independent of $L$, the discretization parameter $h$, and all other geometric information about $\gamma$.  This notation is used mainly \S\ref{ss:geom_assumptions}.  The notation $a \lesssim b$ means that $a \le D b$ with $D$ depending possibly on the same quantities as $C$ and also on $L$, but independent of other geometric information about $\gamma$ and the discretization parameter $h$.  Finally, by $a \preceq b$ we mean $a \le Gb$ with $G$ independent of the discretization parameter $h$ but possibly depending on geometric information about $\gamma$ and the right hand side $f$ in addition to the quantities hidden in $\lesssim$.  

These different levels of precision with respect to constants reflect three different situations:  In the next subsection we are concerned about verifying assumptions which assure the validity of our estimates and thus are as precise as possible concerning constants, including global constants such as the Lipschitz constant $L$.  Global constants such as $L$ are typically hidden as we do in $\lesssim$ when proving residual-type estimates, but it is often desirable to retain relevant {\it local} geometric information such as curvature variation in our estimates so that this information is taken into account when driving refinement in adaptive FEM.  Finally, it becomes highly technical to reflect local geometric information when proving convergence of AFEM as in \cite{BCMMN16}, so it is sometimes convenient to hide such information as we do in $\preceq$.  

\subsection{Geometric resolution assumptions}\label{ss:geom_assumptions}  In order to prove a posteriori error estimates for surface finite element methods and analyze corresponding adaptive algorithms, it is necessary to assume some resolution of the continuous surface $\gamma$ by its discrete approximation.  We thus make three main assumptions concerning the resolution of $\gamma$ by $\Gamma$ and the discrepancy between $\bP$ and $\bP_d$.  These assumptions require more resolution of $\gamma$ than are needed to prove corresponding a posteriori estimates in \cite{BCMMN16} but are necessary to take advantage of the superior properties of the closest point projection $\bP_d$.  

\subsubsection{Restriction on distance between $\Gamma$ and $\gamma$}
\label{ss:res1}
We start by noting that the closest point property \eqref{e:closest} satisfied by the distance function projection implies
$$
| \bP(\bx) - \bP_{d}(\bx)| \leq | \bP(\bx) - \bx + \bx -\bP_{d}(\bx)|
\leq 2 | \bx -\bP(\bx)| \leq 2 | \bX_T(\hat \bx) - \bchi_T(\hat \bx) |
$$
for $\bx \in T$ and $\hat \bx = \bX^{-1}(\bx)$. 
As a consequence {and using that $ \bX_T(\hat \bx) - \bchi_T(\hat \bx) =(I_k \bP-\bP)\circ \overline{\bX}_T$}, we have that for $\bx \in T$ the discrepancy between the two lifts is estimated by the geometric estimator
\begin{equation}\label{e:lift_disc_est}
| \bP(\bx) - \bP_{d}(\bx)| \leq 2 \beta_T.
\end{equation}

Recalling \eqref{N:def} and \eqref{K:def}, we define for $0 <\alpha \le 1$
\begin{equation}\label{Nalpha:def}
\mathcal{N}_\alpha := \{ \bx \in \mathbb{R}^{d+1}: {\rm dist}(\bx, \gamma) \le \alpha \|K\|_{L_\infty(\gamma)}^{-1}\}.
\end{equation}
and we make the assumption that 
\begin{assumption}
\label{N:assumption}
\Gamma \subset \mathcal{N}_{1/2}.
\end{assumption}
This assumption is sufficient to guarantee that all points $\bx$ appearing in our arguments below remain in $\mathcal{N}$, and that 
\begin{equation}\label{e:kappa_equiv}
\kappa_i(\bx) \simeq \kappa_i(\bP_d(\bx))
\end{equation} 
by \eqref{kappas}.  
Combining \eqref{e:closest} with the observation that $|\bx -\bP(x)| \le \beta$, $\bx \in \Gamma$, we see that \eqref{N:assumption} is implied by 
\begin{assumption}\label{Nalpha:condition}
\beta \le \frac{1}{2} \|K\|_{L_\infty(\gamma)}^{-1},  
\end{assumption}
and we shall assume \eqref{Nalpha:condition} throughout.  {This assumption is verifiable a posteriori.  Recall that $K$ is the maximum over $\gamma$ of the principal curvatures $\kappa_i$, which are in turn the eigenvalues of the Weingarten map ${\bf W} = D \bnu$.  In the most practically relevant case $n=2$, the normal $\bnu$ may in turn be computed from the parametric maps $\bXi$ by taking a cross product of the columns of the matrix ${\bf G}$ and normalizing the result.}

While we do not analyze mesh refinement schemes here, the assumption $\Gamma \subset \mathcal{N}_\alpha$ for $\alpha$ sufficiently small guarantees that mesh refinement carried out using the framework of \cite{BCMMN16} will result in a {\it sequence} of meshes lying in $\mathcal{N}_{\alpha'}$ with $\alpha<\alpha'<\frac 1 2$ sufficiently small.  The threshold value of $\alpha$ depends only on $\|I_k\|_{L_\infty \rightarrow L_\infty}$ and is in principal computable.  This property will be necessary to prove convergence of adaptive FEM using our estimators, which we hope to carry out in future work.

\subsubsection{Restriction on the mismatch between $\bP$ and $\bP_d$}
Our second assumption concerns the mismatch between the parametric map $\bP$ used to describe $\gamma$ in our code and the distance function map $\bP_d$.  We assume that 
\begin{assumption}\label{P_Pd:mismatch}
\bP_d \circ \bP^{-1} (\widetilde{T}) \subset \widetilde{\omega}_T, T \in \T.  
\end{assumption}
In view of \eqref{d:beta}, \eqref{e:lift_disc_est}, and the fact that $\gamma$ is a $C^2$ surface, we have that $|\bx-\bP_d \circ \bP^{-1}(\bx)| \lesssim h^2$ independent of $k$.  Because we expect that ${\rm dist} (\widetilde{T}, \partial \widetilde{\omega}_T) \simeq h_T$, \eqref{P_Pd:mismatch} is satisfied for sufficiently fine meshes.  

As we prove in the following proposition, it is possible to check \eqref{P_Pd:mismatch} computationally without accessing $\bP_d$ if the reference patches $\widehat{\omega}_T$ and constant $L$ in \eqref{bi_lipschitz} are known.  Before stating our proposition, we note that because $\gamma$ is a closed surface, there is a constant $c_{\widehat{T}}$ such that 
\begin{equation}\label{d:comega}
{\rm dist}(\widehat{T}, \partial \widehat{\omega}_T) \ge c_{\widehat{T}}, ~T \in \T.
\end{equation}  
\begin{proposition}
\label{checkable2}
Assume that \eqref{N:assumption} is satisfied.  Then \eqref{P_Pd:mismatch} is satisfied if 
\begin{equation} \label{mismatch:computable}
3 \beta_T \le h_T L^{-1} c_{\widehat{T}}, \quad T \in \T.
\end{equation}
\end{proposition} 
\begin{proof}
Let ${\rm dist_g}$ denote the geodesic distance on $\gamma$ and ${\rm dist_E}$ the Euclidean distance.  
Then for $\bx, \by \in \gamma$, $|\bx -\by| ={\rm dist_E}(\bx, \by) \le {\rm dist_g}(\bx, \by)$.  Also, if $\hat{x} \in \widehat{T}$ and $\hat{y} \in \partial \widehat{\omega}_T$, we have from \eqref{bi_lipschitz} that $|\chi_T(\hat{\bx}) -\chi_T(\hat{\by})| \ge h_T L^{-1} | \hat{\bx}-\hat{\by}| \ge h_T L^{-1} c_{\widehat{T}}$.  Thus
\begin{equation}
\label{dist1}
{\rm dist_g} (\widetilde{T}, \partial \widetilde{\omega}_T) \ge {\rm dist_E} (\widetilde{T}, \partial \widetilde{\omega}_T) \ge h_T L^{-1} c_{\widehat{T}}.
\end{equation}

Consider now $\bx \in T \subset \Gamma$.  We wish to control ${\rm dist_g} ( \bP(\bx), \bP_d(\bx))$.  Note that the corresponding Euclidean distance is bounded in \eqref{e:lift_disc_est}.  Let $\tilde{\rho}: [0, 1] \rightarrow \mathbb{R}^{n+1}$ be given by $\tilde{\rho}(s)=s \bP(\bx) + (1-s) \bP_d(\bx)$.
A short computation and \eqref{e:lift_disc_est} imply that for every $s\in [0,1]$ 
\begin{equation} \label{distbound}
{\rm dist_E} (\tilde{\rho}(s), \gamma) \le \frac{1}{2} |\bP(\bx)-\bP_d(\bx)| \le \beta_T, 
\end{equation}
so assumption \eqref{Nalpha:condition} implies that $\tilde{\rho}(s) \in \mathcal{N}_{1/2}$, $0 \le s \le 1$.  We also define $\rho : [0,1] \rightarrow \gamma$ by $\rho(s)=\bP_d( \tilde{\rho}(s))$, which is well defined since $\tilde{\rho}(s) \in \mathcal{N}$.  Thus by \eqref{Nalpha:def}, $|d(\tilde{\rho}(s)) \kappa_i (\tilde{\rho}(s))| \le \frac{1}{2}$, $0 \le s \le 1$.  Let $\rho_j$, $1 \le j \le n+1$ be the components of $\rho$.  Then the length of the curve $\rho$ is given by 
\begin{equation}\label{arclength}
{\rm length} (\rho) = \int_0^1 \left [ \sum_{j=1}^{n+1} \left ( \frac{d \rho_j}{d s} \right )^2 \right ]^{1/2} ds.
\end{equation}
But 
$\frac{d \rho_j} {ds} = [D \bP_d (\tilde{\rho}(s))(\bP(\bx) - \bP_d(\bx))]_j$,
so by \eqref{e:lift_disc_est} we have
\begin{equation}\label{drho}
\begin{aligned}
{\rm length} (\rho) & = \int_0^1 \left | D \bP_d (\tilde{\rho}(s))(\bP(\bx) - \bP_d(\bx)) \right | ds 
\\ & \le  \| \|D \bP_d \circ \tilde{\rho} \|_{\ell_2 \rightarrow \ell_2}  \|_{L_\infty(\tilde{\rho}[0,1])} 2 \beta_T.
\end{aligned}
\end{equation}

Because ${\bW} \nabla d=0$, there is an orthonormal basis $\{ \nabla d, {\bf e}_1,..., {\bf e}_n\}$ of $\mathbb{R}^{n+1}$, and $\bW=\sum_{\ell=1}^n \kappa_\ell {\bf e}_\ell \otimes {\bf e}_\ell$.   From \eqref{e:lift_dist}, we thus have that $D \bP_d = I -\nabla d \otimes \nabla d -d \bW= \sum_{\ell=1}^n (1-d \kappa_\ell) {\bf e}_\ell \otimes {\bf e}_\ell$.  For $\bx \in \mathcal{N}_{1/2}$, \eqref{Nalpha:def} yields
$$\|D \bP_d(\bx) \|_{\ell_2 \rightarrow \ell_2} \le \max_{1 \le \ell \le n} |1-d(\bx) \kappa_\ell(\bx)| \le 1+\frac{1}{2} = \frac{3}{2}, $$
which when combined with the observation that $\tilde{\rho}([0,1]) \subset \mathcal{N}_{1/2}$ and inserted into \eqref{drho} yields
\begin{equation} 
{\rm length} ({\rho}) \le 3 \beta_T.
\end{equation}
In particular,  ${\rm dist_g} (\bP(\bx), \bP_d(\bx)) \le {\rm length}(\rho) \le 3 \beta_T,$
which when combined with \eqref{dist1} yields \eqref{mismatch:computable}.
\end{proof}

\subsubsection{Nondegeneracy of area elements}  The Lipschitz property \eqref{bi_lipschitz} implies that on $\widetilde{\omega}_T$,
\begin{equation}
\label{q:nondegen}
L^{-n} (h_T)^n \lesssim q \lesssim L^n (h_T)^n
\end{equation}
We additionally assume that $q_\Gamma$ and $q_d$ are equivalent to $q$ in the sense that
\begin{assumption}
\label{q:nondegen:assume}
C_1 \le \frac{q}{q_\Gamma}, \frac{q_d}{q_\Gamma} \le C_2.
\end{assumption}

As above, we discuss the possibility of computationally checking the assumption \eqref{q:nondegen:assume}.  From \cite[Lemma 4.1]{BCMMN16} and \cite[Lemma 5.5]{BCMN:Magenes}), we have the following.
\begin{lemma} \label{lem:lambda}
Assume that $\lambda$ is sufficiently small with respect to $L$ and other nonessential constants.  Then
\begin{equation} \label{eq100}
C_1 \le \frac{q}{q_\Gamma} \le C_2, ~~\|\bnu\circ \bP - \bnu_\Gamma\|_{L_\infty(T)} \le C L^{2n-1} \lambda_T, ~T \in \mathcal{T}.
\end{equation}
\end{lemma}
Thus we shall assume throughout that
\begin{assumption}
\label{lambda:assume}
\lambda \hbox{ is sufficiently small that \eqref{eq100} holds.}
\end{assumption}
We do not further discuss the precise threshold in \eqref{lambda:assume}, but intead refer to \cite{BCMMN16}.

In order to provide a condition for the second assumption in \eqref{q:nondegen:assume}, we first  bound $|\bnu-\bnu_\Gamma|=|\bnu\circ \bP_d -\bnu_\Gamma|$.  We will also use these estimates in later sections.   

\begin{lemma} \label{lem:nubound}
Assume that $T \in \mathcal{T}$, and that the assumptions \eqref{N:assumption}, \eqref{P_Pd:mismatch}, and \eqref{lambda:assume} hold.  Then
\begin{eqnarray}
\label{normal:bound}
\| \bnu - \bnu_\Gamma\|_{L_\infty(T)}  \le C[L^{2n-1} \lambda_T + L^2 \|K\|_{L_\infty(\omega_T)} \beta_T],
\\ \label{normal:bound2} \| \bnu - \bnu_\Gamma\|_{L_\infty(T)}^2  \lesssim (\lambda_T)^2 + \|K\|_{L_\infty(\omega_T)} \beta_T.
\end{eqnarray}
\end{lemma}
\begin{proof}
For $\bx \in T$, 
$$
\bnu(\bx) - \bnu_\Gamma(\bx)  = \bnu (\bP_{d}(\bx)) -   \bnu (\bP(\bx))+ \bnu(\bP(\bx)) -  \bnu_\Gamma(\bx).
$$
Employing \eqref{eq100} directly yields 
\begin{equation} \label{bnu1}
\| \bnu\circ \bP - \bnu_\Gamma \|_{L_\infty(T)} \le C L^{2n-1} \lambda_T.
\end{equation}
For the first difference, the assumption \eqref{P_Pd:mismatch} implies  that $\bP(\bx)$ and $\bP_d(\bx)$ lie in the same element patch $\widetilde{\omega}_T$.  Let $\widehat{\omega}_T$ be the reference patch corresponding to $T$, and let $\hat{\bx}=\chi^{-1} (\bP(\bx))$ and $\hat{\bx}_d= \chi^{-1} (\bP_d(\bx))$.  $\widehat{\omega}_T$ is convex, so the line segment between $\hat{\bx}$ and $\hat{\bx}_d$ lies in $\widehat{\omega}_T$ also, and by \eqref{bi_lipschitz} we have $|\hat{\bx}-\hat{\bx}_d| \le L (h_T)^{-1} |\bP(\bx)-\bP_d(\bx)|$.  Using the Lipschitz character of $\chi$, we map this line segment to $\gamma$ and thus obtain a Lipschitz curve ${\bf \xi}$ lying in $\widetilde{\omega}_T$.  Using the arc length formula (cf. \eqref{arclength}) and the Lipschitz bound \eqref{bi_lipschitz}, we find that 
$$
|{\bf \xi}| \le C \int_{\hat{\bx}}^{\hat{\bx}_d} L h_T \le C L^2 |\bP(\bx)-\bP_d(\bx)|.
$$
Letting $D{\bf \xi}$ be the unit tangent vector along the curve ${\bf \xi}$, we have that
\begin{align}
\label{ftoc}
\begin{aligned}
|\bnu (\bP_{d}(\bx)) &  -   \bnu (\bP(\bx))| = \left |\int_{\bf \xi} \bW D {\bf \xi} \right | 
\\ & \le |{\bf \xi}| \| \|\bW\|_{\ell_2 \rightarrow \ell_2}\|_{L_\infty({\bf \xi} )} \le C L^2 |\bP(\bx)-\bP_d(\bx)| \|K\|_{L_\infty(\omega_T)}.
\end{aligned}
\end{align}
Whence, in view of \eqref{e:lift_disc_est}, we get
$$
\|\bnu \circ \bP_{d} -   \bnu \circ \bP\|_{L_\infty(T)}
\le C L^2  \|K \|_{L_\infty(\omega_T)} \beta_T
$$
which along with \eqref{bnu1} yields \eqref{normal:bound}.  
Noting that $\| \bnu \circ \bP- \bnu \circ \bP_d\|_{L_\infty(T)} \le 2$ also trivially yields \eqref{normal:bound2}.  
\end{proof}

The following lemma finally yields a checkable (up to a nonessential constant independent of geometric quantities) condition guaranteeing that the second relationship in \eqref{q:nondegen:assume} holds.
\begin{lemma}  Under the assumptions of Lemma \ref{lem:nubound}, there is a constant $C$ such that
\begin{equation} \label{qdcheck}
\begin{aligned}
\left \| 1-\frac{q_\Gamma}{q_d} \right \|_{L_\infty(T)}  & \le   \frac{C[L^{4n-2} (\lambda_T)^2 + L^4 (\|K\|_{L_\infty(\omega_T)} \beta_T)^2]}{1-C [L^{4n-2} (\lambda_T)^2 + L^4 (\|K\|_{L_\infty(\omega_T)} \beta_T)^2]} 
\\ & ~~~~+ C\|K\|_{L_\infty(\omega_T)} \beta_T.
\end{aligned}
\end{equation}
Thus for $ [L^{4n-2} (\lambda_T)^2 + L^4 (\|K\|_{L_\infty(\omega_T)} \beta_T)^2]$ and $\|K\|_{L_\infty(\omega_T)} \beta_T$ sufficiently small, the second relationship in \eqref{q:nondegen:assume} holds, and
\begin{equation}
\label{qdbound}
\left \| 1-\frac{q_\Gamma}{q_d} \right \|_{L_\infty(T)} \lesssim (\lambda_T)^2 + \|K\|_{L_\infty(\omega_T)}\beta_T.
\end{equation}
\end{lemma}
\begin{proof}
By Proposition 2.1 of \cite{DemlowDziuk:07}, $|\bnu|=|\bnu_\Gamma|=1$, and \eqref{N:assumption}, we have
\begin{equation}
\label{qd:bound1}
\left |\frac{q_d}{q_\Gamma} \right|=|\bnu\cdot \bnu_\Gamma \Pi_{i=1}^n (1-d(x) \kappa_i(x))| \lesssim 1.
\end{equation}
This is the desired upper bound in \eqref{q:nondegen:assume}.  

Again employing Proposition 2.1 of \cite{DemlowDziuk:07}, then using \eqref{kappas}, $|\bnu|=|\bnu_\Gamma|=1$, $|1-\bnu\cdot \bnu_\Gamma|=\frac{1}{2} |\bnu-\bnu_\Gamma|^2$, and $d\kappa_i \le \frac{1}{2} <1$, we have for $x \in \Gamma$ that
\begin{align}
\label{qd:bound}
\begin{aligned}
\Big |1& - \frac{q_\Gamma}{q_d} \Big |=\left |1-\frac{1}{\bnu \cdot \bnu_\Gamma \Pi_{i=1}^n (1-d(x) \kappa_i(x))} \right |
\\& = \left |(1-\frac{1}{\bnu\cdot \bnu_\Gamma}) \frac{1}{\Pi_{i=1}^n (1+d(x) \kappa_i(\bP_d(x)))}- \frac{1}{\Pi_{i=1}^n (1+d(x) \kappa_i(\bP_d(x)))}+1 \right |
\\ & \le C \left | \frac{\bnu\cdot \bnu_\Gamma-1}{\bnu\cdot \bnu_\Gamma-1+1} \right | + C |d (x)K(\bP_d(x))|
\\ & \le C \frac{\frac{1}{2} |\bnu-\bnu_\Gamma|^2}{1-\frac{1}{2} |\bnu-\bnu_\Gamma|^2} + C |d(x)| K(\bP_d(x)).
\end{aligned}
\end{align}

We next estimate $\|d\|_{L_\infty(T)}$.
Recalling the closest point property \eqref{e:closest}, we have for  $\hat{\bx} \in \widehat{T}$ and $\bx = \bX_T(\hat \bx) \in T$
$$
|d(\bx)|  = |\bP_{d}(\bx) - \bx| \leq | \bP(\bx) - \bx| = | \bchi_T(\hat \bx) - \bX_T(\widehat \bx)|
$$
and so
\begin{equation}\label{e:estim_d}
\| d\|_{L_\infty(T)} \leq \|\bchi_T  - \bX_T\|_{L_\infty(\widehat{T})}=\beta_T.
\end{equation}

Inserting the square of \eqref{normal:bound} and \eqref{e:estim_d} into \eqref{qd:bound} and using assumption \eqref{P_Pd:mismatch} to obtain $K(\bP_d(\bx)) \le \|K\|_{L_\infty(\omega_T)}$ completes the proof of \eqref{qdcheck}.  Instead employing \eqref{normal:bound2} in the numerator yields \eqref{qdbound}.  
\end{proof}

We conclude by stating our final assumption:
\begin{assumption}
\label{qd:assume}
\begin{aligned}
L^{4n-2} (\lambda_T)^2  &+ L^4 (\|K\|_{L_\infty(\omega_T)} \beta_T)^2 +\|K\|_{L_\infty(\omega_T)} \beta_T 
\\ & \hbox{ is small enough that \eqref{qdbound} holds.}
\end{aligned}
\end{assumption}



\section{Surface finite element method and a posteriori estimates}\label{S:Laplace-Beltrami}



In this section we first derive a weak formulation of $-\Delta_\gamma u = f$ as well as its finite element counterpart.  We then prove a posteriori estimates under the assumption that $\gamma$ is $C^2$.

\subsection{Variational Formulation and Galerkin Method}\label{S:variational}
%
The space of square integrable functions on $\gamma$, with vanishing mean value is
\[
 L_{2,\#}(\gamma) := \Big\{v \in L_2(\gamma) \;\big|\; \int_\gamma v = 0\Big\}
\]
and its subspace containing square integrable weak derivatives is
\begin{equation*}
 H^1_\#(\gamma) := \Big\{v\in L_{2,\#}(\gamma) \;\big|\; 
         \nabla_{\gamma} v \in [L_2(\gamma)]^{n+1}\}.
\end{equation*}
%

The weak formulation of $-\Delta_\gamma u =f$ is:
For $f\in L_{2,\#}(\gamma)$, seek $u \in H^1_\#(\gamma)$ satisfying
\begin{equation}\label{p:Weak_PdeGm}
 \int_{\gamma} \nabla_{\gamma}u \cdot \nabla_{\gamma} \varphi 
=  \int_\gamma f \, \varphi ,
\qquad \forall \; \varphi \in H^1_\#(\gamma).  
\end{equation}
Existence and uniqueness of a solution $u\in H^1_\#(\gamma)$ is a consequence of
the Lax-Milgram theorem.

On the surface approximation $\T$, the fully discrete problem associated with \eqref{p:Weak_PdeGm} consists 
of finding $U \in \V(\T)$ satisfying
\begin{align}                  \label{FEM:weakform}
\quad \int_{\Gamma} \nabla_{\Gamma} U \cdot
\nabla_{\Gamma} V = \int_{\Gamma} F_\Gamma \, V \qquad \forall \;
V \in\V(\T),
\end{align}
where $F_\Gamma \in L_{2,\#}(\Gamma)$ a suitable approximation of $f$.
The Lax-Milgram theorem again ensures that \eqref{FEM:weakform} admits a unique solution $U$.

Following \cite{BCMN:Magenes,BCMMN16}, we take
\begin{equation}\label{def:F}
  F_\Gamma := f \frac{q}{q_\Gamma},
\end{equation}
which satisfies $\int_\Gamma F_\Gamma = \int_\gamma f = 0$. 
However, in  \cite{BCMMN16, BCMN:Magenes, DemlowDziuk:07}, this choice also leads to no geometric consistency error in the approximation of the right hand side.
In our setting, we recall all the quantities in \eqref{FEM:weakform} are defined using the practical lift $\bP$.
However, in order to take advantage of the super approximation properties of the distance function, they will have to be related to the corresponding quantities defined via the distance lift $\bP_d$. As a consequence, geometric inconsistency will have to analyzed even for the right hand side. 

\subsection{A Posteriori Residual Error Estimators and Oscillations}\label{ss:a-post-def}
In view of the considerations developed in Section~\ref{S:diff-geom}, we define as in \cite{BCMMN16} the PDE error indicator for any $V\in\V(\T)$ by 
\begin{equation*}
\eta_\T (V,F_\Gamma,T)^2 := h_T^2 \| F_\Gamma + \Delta_\Gamma U\|_{L_2(T)}^2 + 
h_T \| \mathcal J(V)\|_{L_2(\partial T)}^2 \qquad\forall\, T\in \T,
\end{equation*}
where for an inter-element face $S \in S_\T$ ($S = T_+ \cap T_-$ with $T_+,T_- \in \T$)
$$
\mathcal J(V)|_S:= \nabla_\Gamma V|_{T_+} \cdot \bn^+ + \nabla_\Gamma V |_{T_-} \cdot \bn^-
$$
and $\bn^{\pm}$ are the outward pointing unit co-normal of $T^\pm$.
The global estimator is the $l_2$ sum of the local quantities, i.e.
$$
\eta_\T (V,F_\Gamma)  := \left( \sum_{T \in \T} \eta_\T (V,F_\Gamma,T)^2 \right)^{1/2}.
$$
We also introduce the {\it oscillation} for any integer $m \geq 1$, $m' \geq 0$,  $V\in \V(\T)$ and $T \in \T$
\begin{equation}\label{d:osc-def}
\begin{split}
& \osc_{\T}(V,f,T)^2  :={} h_T^2 \Big\| 
      (\text{id} - \Pi^2_{m})
      \left(fq + {\text{div}} 
       \big( q_\Gamma {\nabla} V \bG_\Gamma^{-1} \big)\right)
       \Big\|_{L_2(\widehat{T})}^2\\
    & \qquad + h_T
        \Big\| (\text{id} - \Pi^2_{{m'}}) \left(q_\Gamma^+ {\nabla} V^+ (\bG_{\Gamma}^+)^{-1}\widehat{\bn}^+
    + q_\Gamma^-{\nabla} V^- (\bG_{\Gamma}^-)^{-1}\widehat{\bn}^- \right)
    \Big\|_{L_2(\partial \widehat{T})}^2,
\end{split}
\end{equation}
where $\widehat \bn^{\pm}$ is defined according to \eqref{eq:hat_n},
$\bg_\Gamma^\pm$ and $q_\Gamma^\pm=\sqrt{\det \bg_\Gamma^\pm}$
are the first fundamental form and
area element associated to $T^\pm$, and $\Pi^p_{m}$ denotes the best $L_p$-approximation operator onto the 
space $\mathbb P^m$ of polynomials of degree $\leq m$; the domain is implicit from the context.
As for the estimator, we define the global quantity
$$
\osc_{\T}(V,f) = \left( \sum_{T \in \T} \osc_{\T}(V,f,T)^2 \right)^{1/2}.
$$

 In \cite{BCMMN16}, $m=2r-2$ and $m'=2r-1$ is advocated to guarantee that the oscillations decay faster provided the surface parametrization has appropriate piecewise Besov regularity.
In turn, this Besov regularity matches the regularity needed for the adaptive algorithm to deliver optimal rate of convergence when using $\lambda$ as geometric estimator.
In contrast, since our new geometric estimator asymptotically scales like $\mu \approx \lambda^2+\beta$,  less Besov regularity might be required of the parametrization to deliver the same rate of convergence.
In any event, we leave the choice of $m$ and $m'$ open for further studies on the optimality of the proposed algorithm but note that the constants appearing below might depend on these parameters.

\subsection{Geometric Inconsistencies}
The approximation of $u$ in \eqref{p:Weak_PdeGm} by $U$ in \eqref{FEM:weakform} depends on the approximation of $\gamma$ by $\Gamma$ (geometric approximation) and the approximation of $u \circ \bP_{d}$  by $U$ on $\Gamma$ (or equivalently of $u$ by $U\circ \bP_d^{-1}$).  
This is reflected in the error equation we propose to derive now.
We start with the error
\begin{equation}\label{e:error_sup}
\| \nabla_\gamma (u-U\circ \bP_{d}^{-1}) \|_{L_2(\gamma)}  = \sup_{v \in H^1_\#(\gamma), \ \| \nabla_\gamma v\|_{L_2(\gamma)} = 1} \int_\gamma \nabla_\gamma (u-U\circ\bP_{d}^{-1}) \cdot \nabla_\gamma v
\end{equation}
and work on the integral term.
Using the relation \eqref{p:Weak_PdeGm} defining $u$, we write
$$
 \int_\gamma \nabla_\gamma (u-U\circ\bP_{d}^{-1}) \cdot \nabla_\gamma v = \int_\gamma f v - \int_\gamma \nabla_\gamma (U\circ\bP_{d}^{-1}) \cdot \nabla_\gamma v.
$$
In order to use the relation \eqref{FEM:weakform} satisfied by $U$, one needs to write the above two integrals on $\Gamma$ using the change of variables induced by $\bP_{d}$. 
For the first integral this leads to
$$
\int_\gamma f v = \int_\Gamma f_d v_d \frac{q_d}{q_\Gamma},
$$
where $v_d = v \circ \bP_{d}$, $f_d = f \circ \bP_{d}$. 
The change of variables in the second integral is more involved, reading
$$
 \int_\gamma \nabla_\gamma (U\circ\bP_{d}^{-1}) \cdot \nabla_\gamma v =  \int_\Gamma \nabla_\Gamma U \cdot \nabla_\Gamma v_d 
 - \int_\gamma \nabla (U\circ \bP_{d}^{-1})^T \bE_\Gamma^d \nabla_\gamma v
$$
with 
\begin{equation}\label{d:E}
\bE_\Gamma^{d} := \frac{q_\Gamma}{q_d}  \bAg  \left(I-d \bW\right)\bAG \left(I-d \bW\right)\bAg-\bAg.
\end{equation}
Here $\bAg=(I-\bnu \otimes \bnu)$, $\bAG=(I-\bnu_\Gamma \otimes \bnu_\Gamma)$, and the other relevant geometric quantities were defined in Section \ref{ss:distfunc}.  Also, the quantities $d$ and $\bW$ appearing above are evaluated on $\Gamma$.  
Adding and subtracting $\int_\Gamma F v_d$ from the integral term in \eqref{e:error_sup}, we get
\begin{equation}\label{e:error}
\begin{split}
 \int_\gamma \nabla_\gamma (u-U\circ\bP_{d}^{-1}) \cdot \nabla_\gamma v
 =  &\underbrace{\int_\Gamma F v_d -  \int_\Gamma \nabla_\Gamma U \cdot \nabla_\Gamma v_d}_{=: I(v_d)}  + \underbrace{\int_\Gamma f_d v_d \frac{q_d}{q_\Gamma} - \int_\Gamma F v_d}_{=:II(v_d)} \\
 &+ \underbrace{\int_\gamma \nabla (U\circ \bP_{d}^{-1})^T \bE_\Gamma^d \nabla_\gamma v}_{=:III(v)}.
 \end{split}\end{equation}

We can now state our main result.
\begin{theorem}[Efficiency and Reliability] \label{t:main}
Let $\gamma$ be a closed, compact and $C^2$ hypersurface in $\mathbb R^{n+1}$. Let the assumptions \eqref{e:shape_reg}-\eqref{qd:assume} of \S\ref{S:interp--surface} and \S\ref{ss:geom_assumptions} hold.
Then
$$
\begin{aligned}
\| \nabla_\gamma& (u-U\circ \bP_{d}^{-1}) \|_{L_2(\gamma)} 
\\ & \lesssim \eta_\T(U, F_\Gamma)+ \left ( \sum_{T \in \T}\left [ \left (  \lambda_T^2 + \beta_T \|K\|_{L_\infty(\widetilde{\omega}_T)} \right ) \|\nabla_\Gamma U\|_{L_2(T)} +          \beta_T \|f\|_{L_2(T)}  \right]^2 \right )^{1/2}
\\ & \preceq \eta_\T(U,F_\Gamma) + \mu
\end{aligned}
$$
and
$$
\eta_\T(U,F_\Gamma)  \preceq \| \nabla_\gamma (u-U\circ \bP_{d}^{-1}) \|_{L_2(\gamma)} + \osc_\T(U,f) +  \mu. 
$$
\end{theorem}

Its proof directly follows upon estimating $I(v_d)$, $II(v_d)$ and $III(v_d)$.   This is the subject of Sections~\ref{ss:apost}, \ref{ss:geomconst} and \ref{ss:data} below.  


\subsection{Term I: A-posteriori Estimators} \label{ss:apost}
  
The first term in the right hand side of \eqref{e:error} is  estimated by the estimator $\eta_\T$  and is the focus of this section. 
We first take advantage of the relation \eqref{FEM:weakform} satisfied by the finite element solution $U$ to write
\begin{equation}\label{e:termItoBound}
I(v_d) = \int_\Gamma F (v_d-V) -  \int_\Gamma \nabla_\Gamma U \cdot \nabla_\Gamma (v_d-V)
\end{equation}
for any $V \in \V(\T)$.  \eqref{e:termItoBound} is free of geometric approximation and standard arguments \cite{AO00,Vr96} to derive upper and lower bounds for 
the energy error on flat domains can be extended to this case; see \cite{BCMMN16, BCMN:Magenes, DemlowDziuk:07,MMN:11}.  With the notations introduced in Section~\ref{ss:a-post-def}, we have the following a posteriori error estimation result.
Its proof is omitted as it is in essence Lemma~4.5 in \cite{BCMMN16}.

\begin{lemma}[A posteriori upper and lower bounds]\label{L:upper-lower}
\begin{equation}\label{upper}
\sup_{v \in H^1_\#(\gamma), \ \| \nabla_\gamma v\|_{L_2(\gamma)} = 1} I(v \circ \bP_{d}) 
\lesssim \eta_\T(U,F)
\end{equation}
and
\begin{equation}
\label{lower}
\eta_\T(U,F_\Gamma) 
\lesssim  
\sup_{v \in H^1_\#(\gamma), \ \| \nabla v\|_{L_2(\gamma)} = 1} I(v \circ \bP_{d}) + \osc_{\T}(U,f).
\end{equation}
\end{lemma}

\subsection{Term III: Geometric Inconsistency in the Dirichlet form} \label{ss:geomconst}

We discuss in this section the geometric inconsistencies appearing when approximating the PDE \eqref{p:Weak_PdeGm} on an approximated surface $\Gamma$.
As indicated by the error equation \eqref{e:error}, we need to analyze the geometric consistency for the distance lift, which is provided in \cite{DemlowDziuk:07}.
Note that critical to taking advantage of the superconvergent properties of the distance lift is that the consistency error induced by using the practical lift does not appear. 

Following \eqref{e:error}, appropriately bounding Term III reduces to bounding $\|\bE_\Gamma^d\|_{L_\infty}$ defined by \eqref{d:E}.  We carry this out in the following lemma.

\begin{lemma}[Geometric Consistency For the Distance Lift]
If the assumptions \eqref{e:shape_reg}-\eqref{qd:assume} of \S\ref{S:interp--surface} and \S\ref{ss:geom_assumptions} hold, then for  $T = \bX_T(\widehat{T})$ we have
\begin{equation} \label{e:bound:as}
 \| \bE^{d}_\Gamma \|_{L_\infty(T)} \lesssim (\lambda_T)^2+\|K\|_{L_\infty(\omega_T)} \beta_T \preceq  \mu_T,
 \end{equation}
and 
$$
\begin{aligned}
\sup_{v \in H^1_\#(\gamma), \ \| \nabla v\|_{L_2(\gamma)} = 1} III(v)  & \lesssim \left (  \sum_{T \in \mathcal{T}} \left ( \left [(\lambda_T)^2+\|K\|_{L_\infty(\omega_T)} \beta_T \right ] \|\nabla_\Gamma U\|_{L_2(T)} \right ) ^2 \right ) ^{1/2} 
\\ & \preceq  \mu.
\end{aligned}
$$
\end{lemma}
\begin{proof}
Because $\bAg$ and $\bAG$ are projections, $\|\bAg\|_{\ell_2 \rightarrow \ell_2} \le 1$ and $\|\bAG \|_{\ell_2 \rightarrow \ell_2} \le 1$.  In addition, $\|\bW\|_{\ell_2 \rightarrow \ell_2} \le K$.  By \eqref{N:def} $\|d\bW\|_{\ell_2 \rightarrow \ell_2} \le d K \le 1$ and $\|I-d\bW\|_{\ell_2 \rightarrow \ell_2} \lesssim 1$.  Using these facts easily yields that for $\bx \in \Gamma$,
\begin{equation} \label{matbound1}
\|\bAg (I-d\bW)\bAG (I-d\bW)\bAg \|_{\ell_2 \rightarrow \ell_2}  \lesssim 1.
\end{equation}
Elementary calculations and the assumption \eqref{N:assumption} also yield that
\begin{equation} \label{matbound2}
\begin{aligned}
\|\bAg -\bAg & (I-d\bW) \bAG (I-d\bW)\bAg \|_{\ell_2 \rightarrow \ell_2}  
\\ & \lesssim \|\bAg -\bAg\bAG\bAg\|_{\ell_2 \rightarrow \ell_2} + |d(\bx)|K(\bx)
\\ & =\|(\bnu_\Gamma-(\bnu_\Gamma \cdot \bnu) \bnu) \otimes (\bnu_\Gamma-(\bnu_\Gamma \cdot \bnu) \bnu)\|_{\ell_2 \rightarrow \ell_2} + |d(\bx)|K(\bx)
\\ & \lesssim |\bnu-\bnu_\Gamma|^2 + |d(\bx)| K(\bP_d(\bx)), 
\end{aligned}
\end{equation}
where we used the property $|1-\bnu\cdot \bnu_\Gamma|=\frac{1}{2} |\bnu-\bnu_\Gamma|^2$ and \eqref{e:kappa_equiv} to derive the last inequality.  

Recalling the definition \eqref{d:E} yields
$$
\begin{aligned}
\| \bE^{d}_\Gamma  \|_{\ell_2 \rightarrow \ell_2} &  \le  \left | 1-\frac{q_\Gamma}{q_d} \right | \| \bAg (I-d\bW) \bAG (I-d\bW)\bAg \|_{\ell_2 \rightarrow \ell_2} 
\\ & ~~~+ \|\bAg -\bAg (I-d\bW) \bAG (I-d\bW)\bAg \|_{\ell_2 \rightarrow \ell_2}.
\end{aligned}
$$
Inserting \eqref{normal:bound2} into \eqref{matbound2}, applying \eqref{P_Pd:mismatch} to obtain $K(\bP_d(\bx)) \le \|K\|_{L_\infty(\widetilde{\omega}_T)}$, and then gathering the result along with \eqref{matbound1} and \eqref{qdcheck} into the above inequality yields \eqref{e:bound:as}.  
%
The second assertion follows from the definition of $III(v)$ in \eqref{e:error} and a Cauchy Schwarz inequality.
\end{proof}

\subsection{Term II: Data Inconsistencies}\label{ss:data}

We now focus on the second term in \eqref{e:error}:

$$
II(v_d) = \int_\Gamma f_d v_d \frac{q_d}{q_\Gamma} - \int_\Gamma F v_d.
$$
where $F =  (f \circ \bP) \frac{q}{q_\Gamma}$. 

Recall that we do not assume that the distance function, and therefore $\bP_{d}$, is accessible to the user. Otherwise, the choice $\bP = \bP_{d}$ would lead to $II=0$.
The mismatch between  $\bP(\bx)$ and $\bP_{d}(\bx)$ is reflected in a non vanishing term $II$ which must be accounted for.

Lemma~\ref{l:geom_err_func} provides an estimate for the effect of the discrepancy $\tilde v-\tilde v \circ \bP_d \circ \bP^{-1}$ for $\tilde v \in H^1(\gamma)$ in $L_2(\gamma)$.  Its proof uses standard properties of mollifiers. Given $ v \in H^1(\mathbb{R}^n)$, there exists a function $v^\epsilon$ such that for every $U \subset \mathbb{R}^n$
\begin{align}
\label{e:molifier}
\| v - v^\epsilon \|_{L_2(U)} \lesssim \epsilon  | v |_{H^1( U+ B(0,\epsilon))}\\
\textrm{and} \qquad | v^\epsilon |_{W^1_\infty(U)} \lesssim \epsilon^{-n/2}  |v^\epsilon |_{H^1(U+ B(0,\epsilon))}. \label{e:molifier2}
\end{align}

We are now in position to estimate in $L_2$ the effect of the mismatch of $\bP_d \not = \bP$ on a $H^1$ function.

\begin{lemma}[Geometric Error in Function Evaluation]\label{l:geom_err_func}
Assume that assumptions \eqref{e:shape_reg}-\eqref{qd:assume} of Section \S\ref{S:interp--surface} hold.
Then for $\tilde{w} \in H^1(\gamma)$ and $T \in \T$ we have
$$\|\tilde{w}-\tilde{w} \circ \bP_d \circ \bP^{-1}\|_{L_2(\widetilde{T})} \lesssim \beta_T \|\tilde{w}\|_{H^1(\tilde{\omega}_T)}.$$
\end{lemma}

\begin{proof}
For notational ease, let $\psi=\bP_d \circ \bP^{-1}$.  Fix $T \in \T$, let $\hat{w}(\hat{x})=\tilde{w} \circ \chi_T(\hat{x})$ for $\hat{x} \in \widehat{\omega}_T$, and let $\hat{\psi}=(\chi_T)^{-1} \circ \psi \circ \chi_T$.  Then by change of variables and \eqref{q:nondegen},
$$\|\tilde{w}-\tilde{w}\circ \psi\|_{L_2(\widetilde{T})} \lesssim h_T^{n/2} \|\hat{w}-\hat{w} \circ \hat{\psi}\|_{L_2(\widehat{T})}.$$
The assumption $\bP_d \circ \bP^{-1} (\widetilde{T}) \subset \widetilde{\omega}_T$ given in \eqref{P_Pd:mismatch} is equivalent to $\hat{\psi}(\widehat{T}) \subset \widehat{\omega}_T$ and is sufficient to ensure that the quantity on the right hand side is well-defined.

Note that $\hat{w}$ is defined on the reference patch $\widehat{\omega}_T$.  There is a universal extension operator mapping $H^1(\widehat{\omega}_T)$ to $H^1(\mathbb{R}^n)$ which is bounded both in $L_2$ and in the $H^1$-seminorm.  
We thus may assume that $\hat{w}$ is defined and bounded in $H^1$ on all of $\mathbb{R}^n$ and that 
$$|\hat{w}|_{H^1(\mathbb{R}^n)} \lesssim |\hat{w}|_{H^1(\widehat{\omega}_T)}.$$
Given $\epsilon>0$, we denote by $\hat{w}^\epsilon$ the standard mollification of $\hat{w}$ with radius $\epsilon$.  We may then apply \eqref{e:molifier} and \eqref{e:molifier2} to $\hat{w}$ without restriction on $\epsilon$, and
$$\|\hat{w}-\hat{w} \circ \hat{\psi}\|_{L_2(\widehat{T})} \lesssim \|\hat{w}-\hat{w}^\epsilon\|_{L_2(\widehat{T})}+ \|\hat{w}^\epsilon-\hat{w}^\epsilon \circ \hat{\psi}\|_{L_2(\widehat{T})} + \|\hat{w}^\epsilon\circ \hat{\psi}-\hat{w} \circ \hat{\psi}\|_{L_2(\widehat{T})}.$$

We first compute using \eqref{e:molifier} that 
$$\|\hat{w}-\hat{w}^\epsilon \|_{L_2(\widehat{T})} \lesssim \epsilon |\hat{w}|_{H^1(\mathbb{R}^n)} \lesssim |\hat{w}|_{H^1(\widehat{\omega}_T)}.$$
Similarly, after applying a change of variables formula using \eqref{q:nondegen} and \eqref{q:nondegen:assume} and applying the restriction $\hat{\psi}(\widehat{T}) \subset \widehat{\omega}_T$ we find that
$$\|(\hat{w}^\epsilon -\hat{w}) \circ \hat{\psi}\|_{L_2(\widehat{T})} \lesssim \|\hat{w}^\epsilon-\hat{w}\|_{L_2(\widehat{\omega}_T)} \lesssim \epsilon |\hat{w}|_{H^1(\widehat{\omega}_T)}.$$

We finally bound the term $\|\hat{w}^\epsilon-\hat{w}^\epsilon \circ \hat{\psi}\|_{L_2(\widehat{T})}$.  Let $\{x_\ell\}$ be a lattice on $\mathbb{R}^n$ with minimum distance between $x_\ell$ and $x_j$ ($\ell \neq j$) equivalent to $\epsilon$ and such that $\{B_\epsilon(x_\ell)\}$ covers $\mathbb{R}^n$.  The set $\{B_{M\epsilon}(x_\ell)\}$ then has finite overlap for any $M>0$, with the maximum cardinality of the overlap depending on $M$.  Let also $\epsilon = \|I-\hat{\psi}\|_{L_\infty(\widehat{T})}$.  Then applying \eqref{e:molifier2}, we find that
\begin{align}
\begin{aligned} 
\|\hat{w}^\epsilon & -\hat{w}^\epsilon \circ \hat{\psi}\|_{L_2(\widehat{T})}^2 \le \sum_{\ell} \|\hat{w}^\epsilon-\hat{w}^\epsilon \circ \hat{\psi}\|_{L_2(B_\epsilon(x_\ell) \cap \widehat{T})}^2
\\ & \lesssim \epsilon^{n} \sum_{\ell} \|\hat{w}^\epsilon-\hat{w}^\epsilon \circ \hat{\psi}\|_{L_\infty(B_\epsilon(x_\ell) \cap \widehat{T})}^2 
 \lesssim \epsilon^{n+2}  \sum_{\ell} |\hat{w}^\epsilon|_{W_\infty^1(B_{2\epsilon}(x_\ell))}^2
\\ & \lesssim \epsilon^2 \sum_{\ell} |\hat{w}|_{H^1(B_{3\epsilon}(x_\ell))}^2
 \lesssim \epsilon^2 |\hat{w} |_{H^1(\mathbb{R}^n)}^2
\lesssim \epsilon^2 |\hat{w}|_{H^1(\widehat{\omega}_T)}^2.
\end{aligned}
\end{align}

We finally compute using the bi-lipschitz character of $\chi$ that 
\begin{align}
\begin{aligned} 
\epsilon=\|I&-\hat{\psi}\|_{L_\infty(\widehat{T})}   = \|(\chi_T)^{-1} \circ(I-\psi) \circ \chi_T \|_{L_\infty(\widehat{T})} 
\\ & \le Lh_T^{-1} \|(I-\psi) \circ \chi_T \|_{L_\infty(\widehat{T})} = L h_T^{-1}\|I-\bP_d \circ \bP^{-1} \|_{L_\infty(\widetilde{T})}.
\end{aligned}
\end{align} 
Recalling \eqref{e:lift_disc_est} then yields
$$\epsilon \lesssim h_T^{-1} \beta_T.$$
Again performing a change of variables yields
$$\begin{aligned}
\|\tilde{w}-\tilde{w} \circ \psi \|_{L_2(\widetilde{T})}^2 &\lesssim h_T^n \|\hat{w}-\hat{w} \circ \hat{\psi}\|_{L_2(\hat{T})}^2 \lesssim h_T^n \epsilon^2 |\hat{w}|_{H^1(\hat{\omega}_T)}^2 
\\ & \lesssim h_T^n h_T^{-2} \beta_T^2 h_T^{2-n} |\tilde{w}|_{H^1(\widetilde{\omega}_T)}^2 =\beta_T^2 |\tilde{w}_{H^1(\widetilde{\omega}_T)}^2.  
\end{aligned}
$$
\end{proof}

We now return to term II and estimate the discrepancy between the theoretical quantity 
$$
\int_\Gamma (f  v) \circ \bP_{d} \frac{q_d}{q_\Gamma} 
$$
and the practical quantity
$$
\int_\Gamma (f \circ \bP) (v \circ \bP_{d}) \frac{q}{q_\Gamma}.
$$
This is the purpose of the next lemma.

\begin{lemma}[Data Geometric Consistency Error]\label{lem:data_cons}
For every $v \in H^1(\gamma)$ and $f \in L_2(\gamma)$ we have
\begin{align}\label{data_cons}
\begin{aligned}
\left| \int_{\Gamma} (f v) \circ \bP_{d} \frac{q_d}{q_\Gamma} 
-
\int_{\Gamma} (f \circ \bP) (v \circ \bP_{d}) \frac{q}{q_\Gamma} \right| & \lesssim 
\sum_{T \in \T} \beta_T  \|f\|_{L_2(T)} \|v\|_{H^1(\omega_T)} 
\\ & \lesssim (\sum_{T \in \T} (\beta_T)^2 \|f\|_{L_2(T)}^2)^{1/2} \|v\|_{H^1(\gamma)}
\\ & \lesssim \| f \|_{L_2(\gamma)} \| v \|_{H^1(\gamma)} \beta.
\end{aligned}
\end{align}
Thus 
$$
\sup_{v \in H^1_\#(\gamma), \ \| \nabla v\|_{L_2(\gamma)} = 1} II(v \circ \bP_{d}) \lesssim \left (  \sum_{T \in \T} \left [ \| f \|_{L_2(\tilde{T})} \beta_T \right ] ^2 \right)^{1/2} \preceq \beta.
$$
\end{lemma}
\begin{proof}
We write both integrals on $\gamma$ using the lifts $\bP_{d}$ and $\bP$ respectively:
$$
 \int_{\Gamma} (f v) \circ \bP_{d} \frac{q_d}{q_\Gamma} = \int_\gamma f v , \qquad 
\int_{\Gamma} (f \circ \bP) (v \circ \bP_{d}) \frac{q}{q_\Gamma} 
=
\int_{\gamma} f   (v \circ \bP_{d} \circ \bP^{-1})
$$
so that
\begin{equation*}
\begin{split}
\left| \int_{\Gamma} (f v) \circ \bP_{d} \frac{q_d}{q_\Gamma} 
-
\int_{\Gamma} (f \circ \bP) (v \circ \bP_{d}) \frac{q}{q_\Gamma} \right| & \lesssim 
\sum_{T \in \T} \| f \|_{L_2(T)} \| v -  v \circ \bP_{d} \circ \bP^{-1}\|_{L_2(T)}.   
\end{split}
\end{equation*}
The first ``$\lesssim$'' in \eqref{data_cons} follows upon invoking Lemma~\ref{l:geom_err_func}, the second upon using Cauchy-Schwarz, and the third upon applying finite overlap of $\{\widetilde{\omega}_T\}_{T \in \T}$.  
The second assertion in Lemma \ref{lem:data_cons} follows from the definition of $II(v_d)$ in \eqref{e:error} and the application of a Cauchy-Schwarz inequality.
\end{proof}

\section{Numerical Illustrations} \label{s:numerics}

We present several examples illustrating the benefits when using the proposed geometric estimator $\mu$ instead of $\lambda$.  We consider the adaptive algorithm proposed in \cite{ BCMMN16, BCMN:Magenes} and recalled now:

\medskip
{\bf AFEM:}
Given an initial subdivision $\T_0$ and parameters $\varepsilon_0>0$, $0<\rho<1$, and
$\omega>0$, set $k=0$.
\begin{algotab}
  \> \>1. $\T_k^+ = \ADAPTSURF (\T_k,\omega \eps_k)$\\
  \> \>2. $[U_{k+1},\T_{k+1}] = \ADAPTPDE (\T_k^+,\eps_k)$ \\
  \> \>3. $\eps_{k+1} = \rho \eps_k$; $k = k+1$ \\  
  \> \>4. go to 1.
\end{algotab}
\noindent

The procedure $\ADAPTSURF (\T_k,\omega \eps_k)$ targets a better approximation of the surface and produces a finer subdivision $\T_k^+$ such that the geometric estimator $\lambda_{\widehat{\T}_k^+}$ reaches a value below $\omega \eps_k$. 
This is performed using a greedy algorithm, i.e. the elements of the subdivisions are successively refined until the geometric estimator on each element is below the targeted tolerance $\eps$:

\begin{algotab}
\> $\T_* = \ADAPTSURF (\T,\eps)$ \\
\> \> 1. if  $\mathcal M :=\{T\in \T \, : \,\lambda_{\T}(\gamma,T)>\eps \} =\emptyset$ \\
\>  \>  \> \> return($\T$) and exit \\
\> \> 2. $\T = \REFINE(\T,\mathcal M)$\\
\> \> 3. go to 1.
\end{algotab}

The procedure $\ADAPTPDE (\T_k^+,\eps_k)$ instead consists of several iterations of the adaptive loop 
\begin{equation}\label{e:SEMR}
\SOLVE \rightarrow \ESTIMATE \rightarrow \MARK \rightarrow \REFINE
\end{equation}
aimed towards reducing the PDE error until the finer subdivision / approximate surface $\T_{k+1} / \Gamma_{k+1}$ are such that the associated finite element solution $U_{k+1}$ satisfies $\eta_{\T_{k+1}}(U_{k+1},F_{\Gamma_{k+1}}) \leq \eps_k$.

In both cases, the resulting subdivision is conforming upon refining additional elements without increasing the overall complexity.
We refer to \cite{BCMMN16, BCMN:Magenes} for additional details.  

While our theory above is presented for subdivision made of simplices, our implementation is based on the \textrm{deal.ii} library \cite{BHK:07} and uses quadrilaterals and finite element spaces $Q_s$ (instead of $\mathcal P_s$) based on polynomials of degree at most $s$ in each direction on the  reference hypercube.
In addition, in both of our examples below the surface is not closed, that is, $\partial \gamma \neq \emptyset$.
Our arguments extend to these different situations under some further assumptions, most notably that $\bP(\Gamma)=\bP_d(\Gamma)=\gamma$, which is indeed the case for the two examples below.

\subsection{PDE error driven on a half sphere}

The first example consists of a case where the geometry, i.e. the surface, is smooth and therefore the dominant term estimating the error (see Theorem~\ref{t:main}) is the PDE estimator $\eta$. 

We consider the half sphere $\gamma = \mathbb S^2 \cap \left\lbrace \bx \geq 0 \right\rbrace$.
We also choose $\omega=1$ and $\rho=\frac 1 2$.  {The exact solution is given in spherical coordinates by $u=\sin(\phi) \sin(\theta)^{3/5}$, with $f=-\Delta_\gamma u$.  Thus $u$ has a singularity at the north pole similar to a corner singularity on a nonconvex polygonal domain, and a graded mesh is necessary to recover optimal convergence.}  For the approximation parameters, we take $r=2$ and $k=1$, which asymptotically yields decays $\eta \sim N^{-1}$, $\lambda \sim N^{-\frac 12}$ and $\mu \sim N^{-1}$ when using $N$ degrees of freedom in a mesh.  This is optimal for the given quantity.  The striking difference in the performance of the adaptive routine when using the BCMMN estimator and the proposed one is illustrated in Figure~\ref{fig:halfsphere} and Table~\ref{tab:halfsphere}.   The meshes indicate that the BCMMN estimator equidistributes far more degrees of freedom across the sphere in order to resolve geometry, while the BD estimator concentrates refinement at the north pole in order to control the singularity in the PDE solution.  The BCMMN estimator quantitatively yields suboptimal convergence $N^{-1/2}$ of the overall error due to overrefinement based on the geometric estimator.  In contrast, the proposed BD algorithm yields optimal $N^{-1}$ convergence. 

\begin{figure}
\begin{center}
\begin{tabular}{cc}
\includegraphics[width=0.3\textwidth]{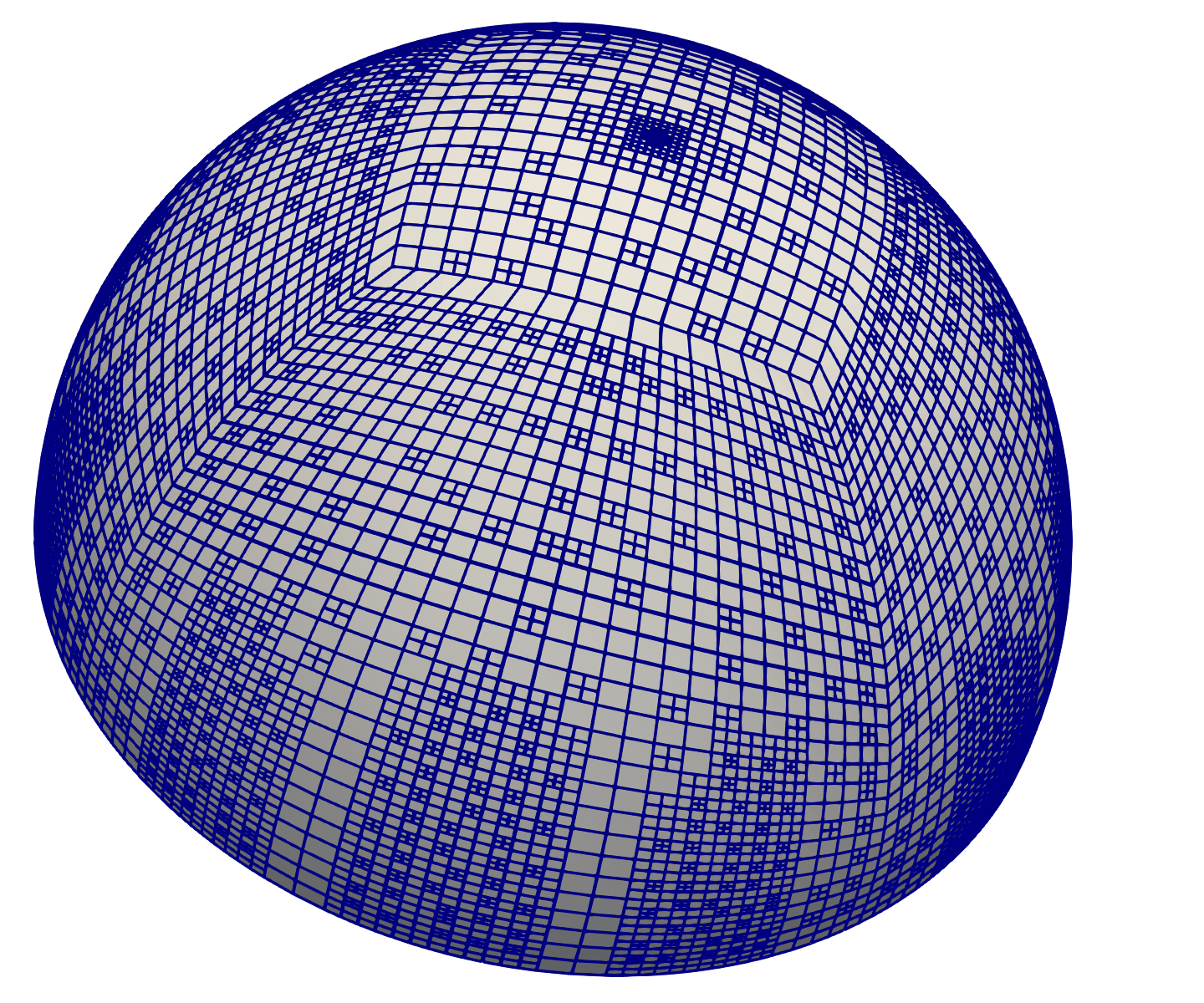} & \multirow{2}{*}[16ex]{\includegraphics[width=0.63\textwidth]{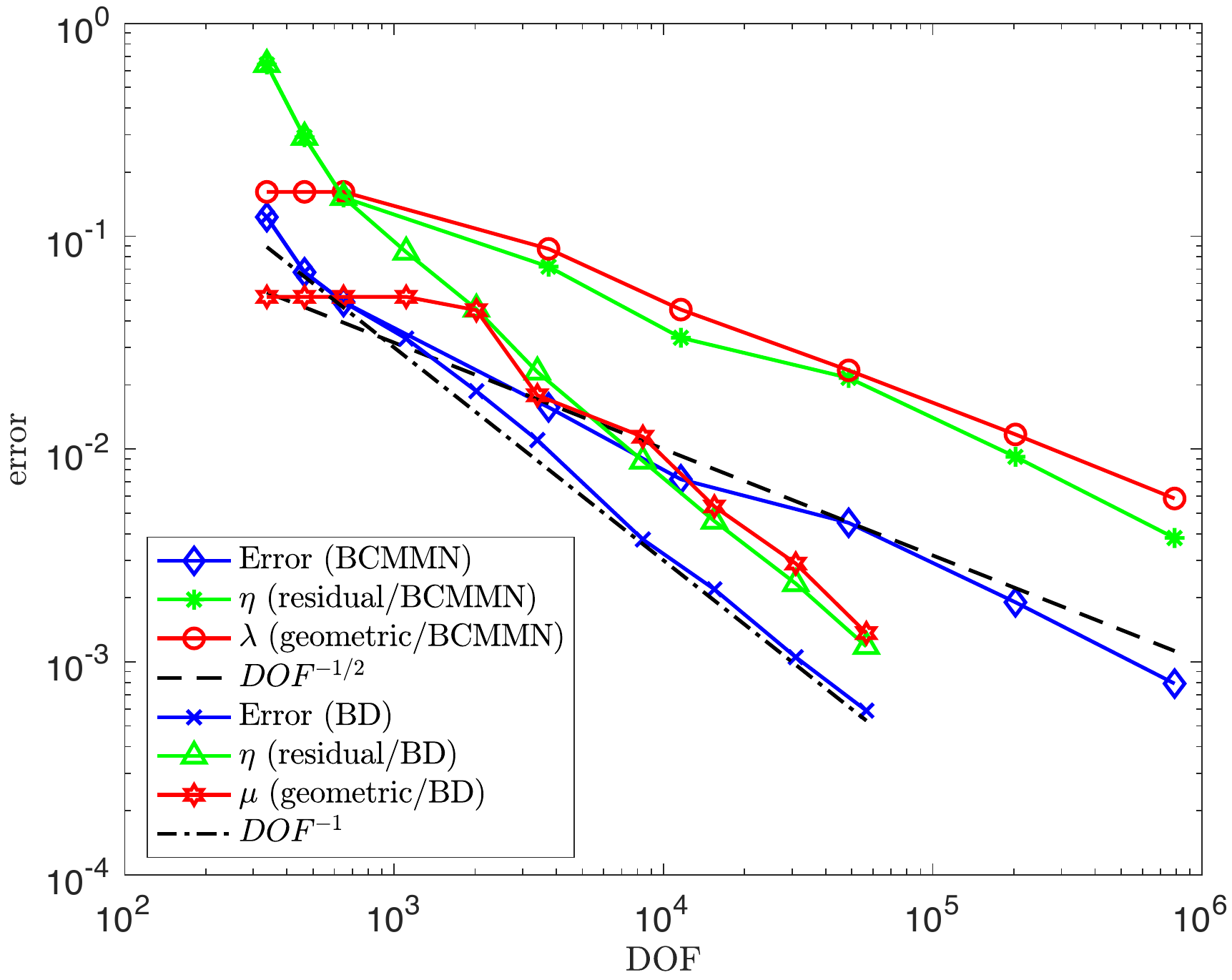}}\\
\includegraphics[width=0.3\textwidth]{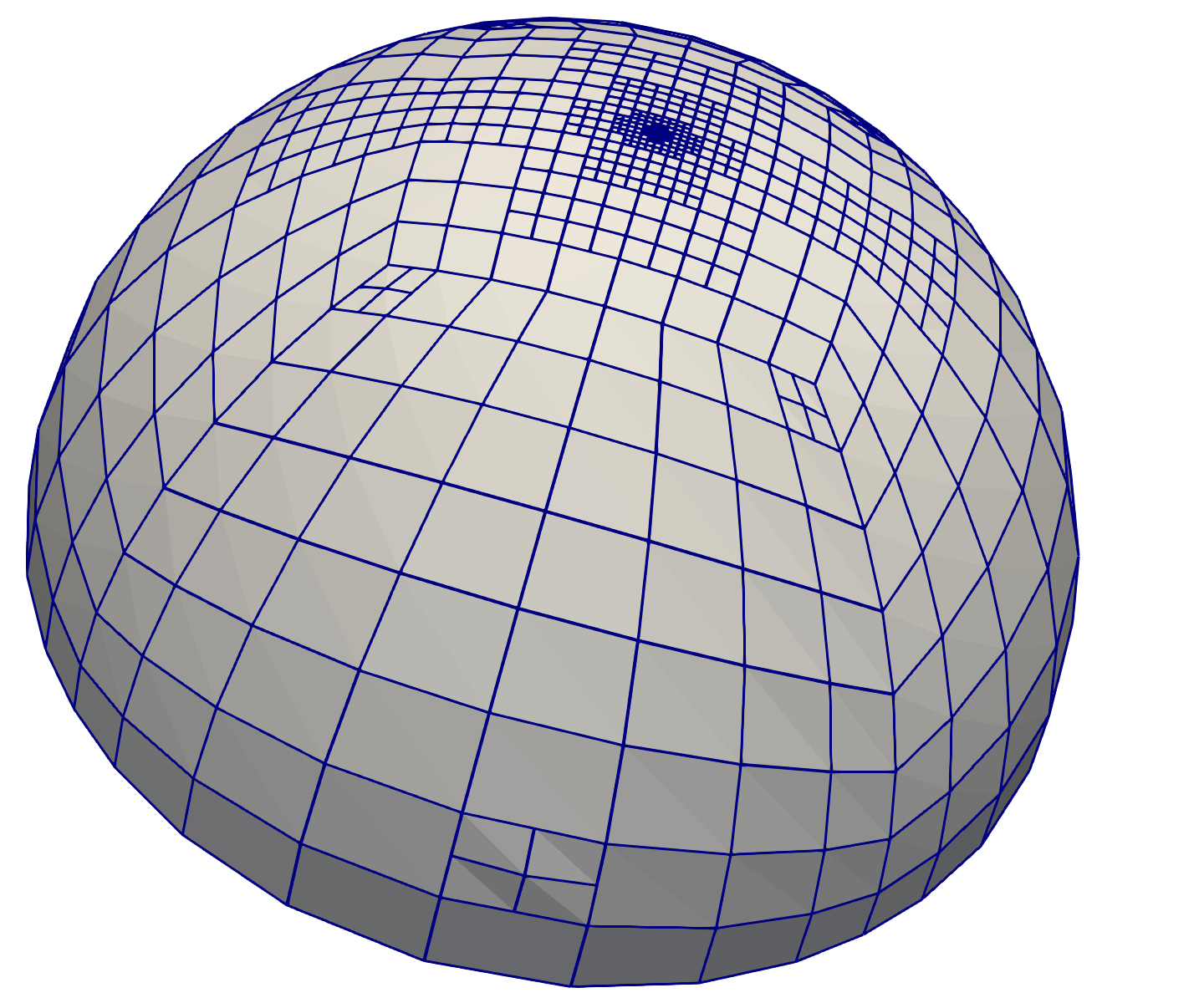} &\\
\end{tabular}
\end{center}
\caption{Smooth surface: Comparison between the algorithm based on the geometric estimators $\lambda$ (BCMMN)  and $\mu$ (BD).
The subdivisions constructed to achieve an error tolerance $\eps = 3/2^8$ are provided in the left column for (top) BCMMN and (bottom) BD. 
When $\gamma$ is the upper half sphere, the adaptive algorithm based on $\mu$ produces a geometric refinement drastically less pronounced in order to achieve the targeted tolerance. In fact, the BCMMN algorithm yields an overrefinement for the geometry in this context.
The right plot provide the evolution of the errors and estimators during the adaptive algorithms. 
The BD algorithm yields optimal error decay $\textrm{DOF}^{-1}$ unlike the BCMMN version which exhibits a convergence order of $\textrm{DOF}^{-1/2}$. }\label{fig:halfsphere}
\end{figure}

\begin{table}
\begin{tabular}{c|cc}
& BCMMN &  BD \\
\hline
Total no of PDE refinement loops \eqref{e:SEMR} & 9  &  10  \\
Total no of Geometric refinement loop in \ADAPTSURF & 1388 &  7\\
Total no of degree of freedom to achieve tolerance $\eps=3/2^8$ & 
 48689  &  3403
\end{tabular}
\caption{Comparison of the BCMMN \cite{BCMMN16, BCMN:Magenes} and the proposed estimator when used in an adaptive finite element loop.} \label{tab:halfsphere}
\end{table}

\subsection{PDE error driven on a $C^{2,\alpha}$ surface}

We now turn our attention to a case where the geometric error plays an important role due to limited regularity of $\gamma$.
For this, we consider the surface $\gamma$ given as the graph of
$$
z(x,y) = \left( \frac 3 4 - x^2 - y^2 \right)_{+}^{2+\alpha}
$$
for $\alpha = \frac 2 5$ and where $(x,y) \in (0,1)^2$ which implies $\gamma \in C^{2,\alpha}$.  Here we also denote by $(v)_+$ the positive portion of $v$.  
In this case again, the parametrization of $\gamma$ is given by $z$.
We set $f = 1$ and impose vanishing Dirichlet boundary conditions on $\partial \gamma$. 

We execute both algorithms up to a final tolerance $\eps= 5 \times 10^{-7}$ and again with 
$\omega=1$ and $\rho=\frac 1 2$. 
For the BCMMN algorithm we set $r=k=3$ while $r=3$, $k=2$ is chosen for the BD algorithm.
The exact solution $u$ is unknown but we expect that the PDE estimator will behave like $N^{-3/2}$ when using $N$ degrees of freedom and optimal meshes.
This is indeed the case for the BD algorithm as illustrated in Figure~\ref{fig:C2alpha}, where in addition to the values of the estimator for different values of $N$, the subdivisions constructed in both cases to guarantee and error smaller that $\eps$ are reported. Note that the BCMMN algorithm seems to exhibit a slightly suboptimal error decay, likely due to inability to resolve the geometric error notion controlled by $\lambda$ with rate $N^{-3/2}$.

{We finally note that the convergence order of $DOF^{-1}$ observed under uniform refinement above indicates that roughly $u \in H^3(\gamma)$.  To see this, let $\bXi=(x,y,z(x,y))$ be the lift from $\Omega=(0,1)^2$ to $\gamma$, and let $\hat{u}=u \circ \bXi^{-1}$.  According to \eqref{lb_def}, $\hat{u}$ solves the elliptic PDE $-{\rm div} (q \nabla \hat{u} {\bf g}^{-1}) = q$ in $\Omega$.  Note that $q=1$ and the coefficient $q {\bf g}^{-1}=I$ outside of the circle $x^2+y^2=1$, and these quantities vary inside of the circle.  At the three corners of $\Omega$ where $z(x,y)=0$ we are thus solving $-\Delta \hat{u}=1$, and standard regularity theory for polygonal domains indicates that in those regions $u \in H^s$ for any $s<3$ (cf. \cite[Theorem 14.6]{Da88}).  At the origin (the corner of $\Omega$ at which $z(x,y)$ varies; cf. the illustration in Figure \ref{fig:C2alpha}) $\hat{u}$ solves an elliptic PDE with smooth coefficient.  Establishing the regularity of $\hat{u}$ near this corner is more complicated due to the presence of the coefficient, but the coefficient is smooth and isotropic at the origin and so it is plausible that the regularity is the same here as at the other corners.  In addition, away from the boundary $\hat{u}$ solves a PDE with smooth coefficient.  In particular, it is possible to calculate using the definitions in Section \ref{S:diff-geom} that at a minimum the area element $q \in C^1(\Omega)$ (and thus also $q \in H^1(\Omega)$) and that the coefficient $q {\bf g}^{-1} \in C^2(\Omega)$.  In fact, each appears to possess at least fractionally more smoothness than this.  Standard elliptic regularity theory thus yields also that at the least  $\hat{u} \in H^3({\rm int}(\Omega))$.  Thus elliptic corner singularities and not the regularity of $\gamma$ appear to place the heaviest restriction on the regularity of $\hat{u}$ and thus on convergence rates under quasiuniform refinement.  On the other hand, these corner singularities have infinite smoothness in the context of adaptive refinement and are not responsible for the limited convergence rate observed when using BCMMN refinement.}

\begin{figure}
\begin{center}
\begin{tabular}{cc}
\includegraphics[width=0.3\textwidth]{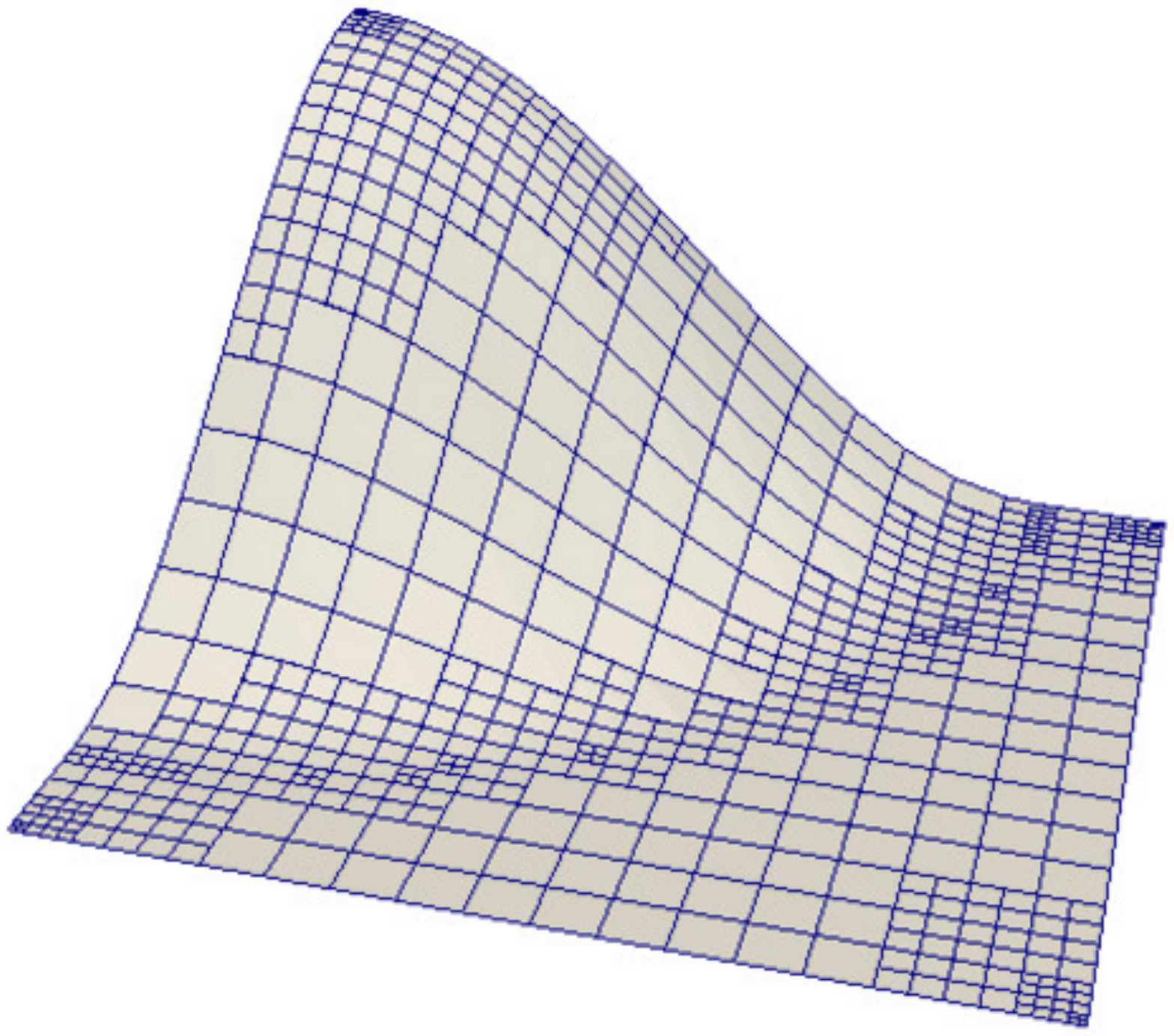} & \multirow{2}{*}[30ex]{\includegraphics[width=0.63\textwidth]{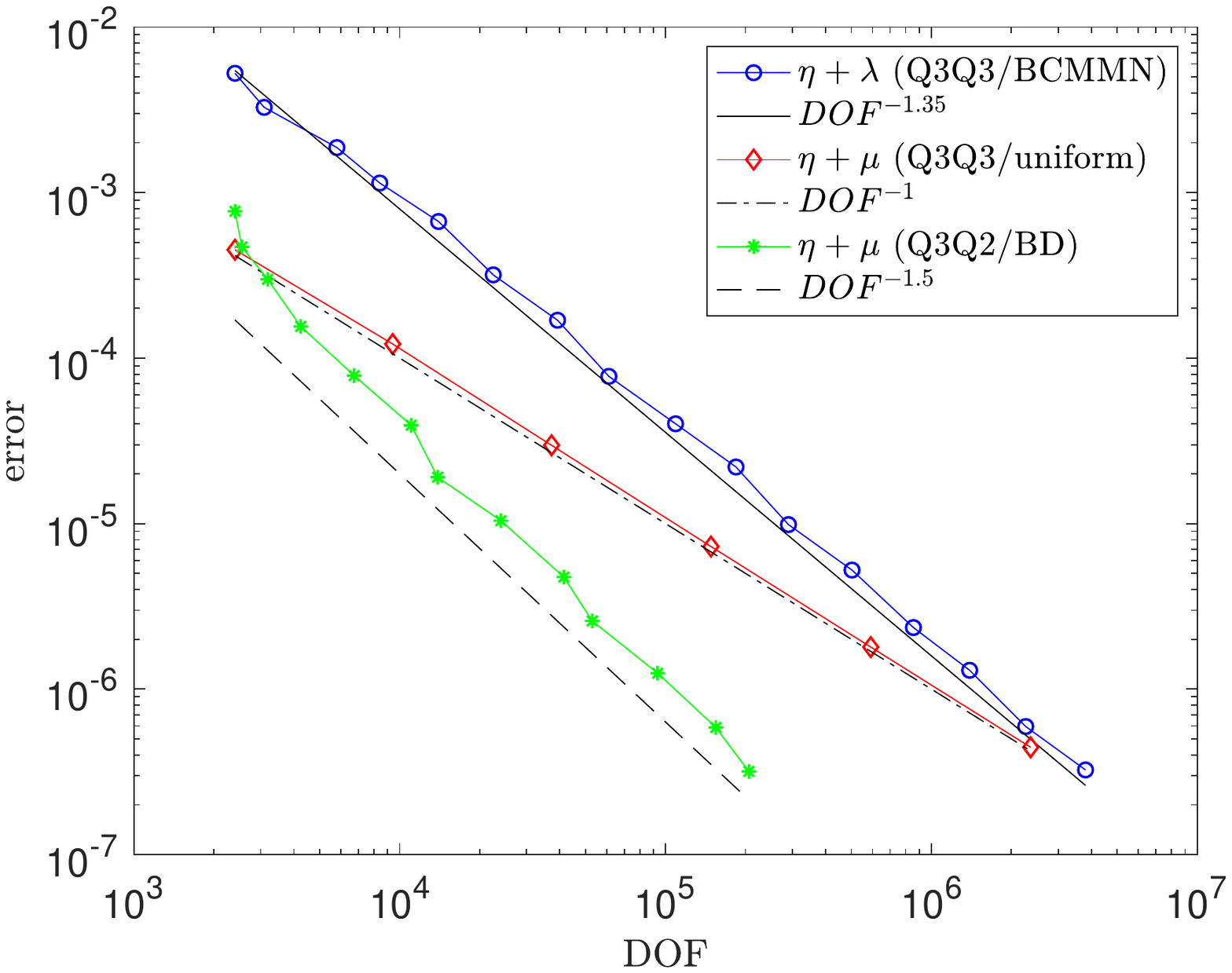}}\\
\includegraphics[width=0.3\textwidth]{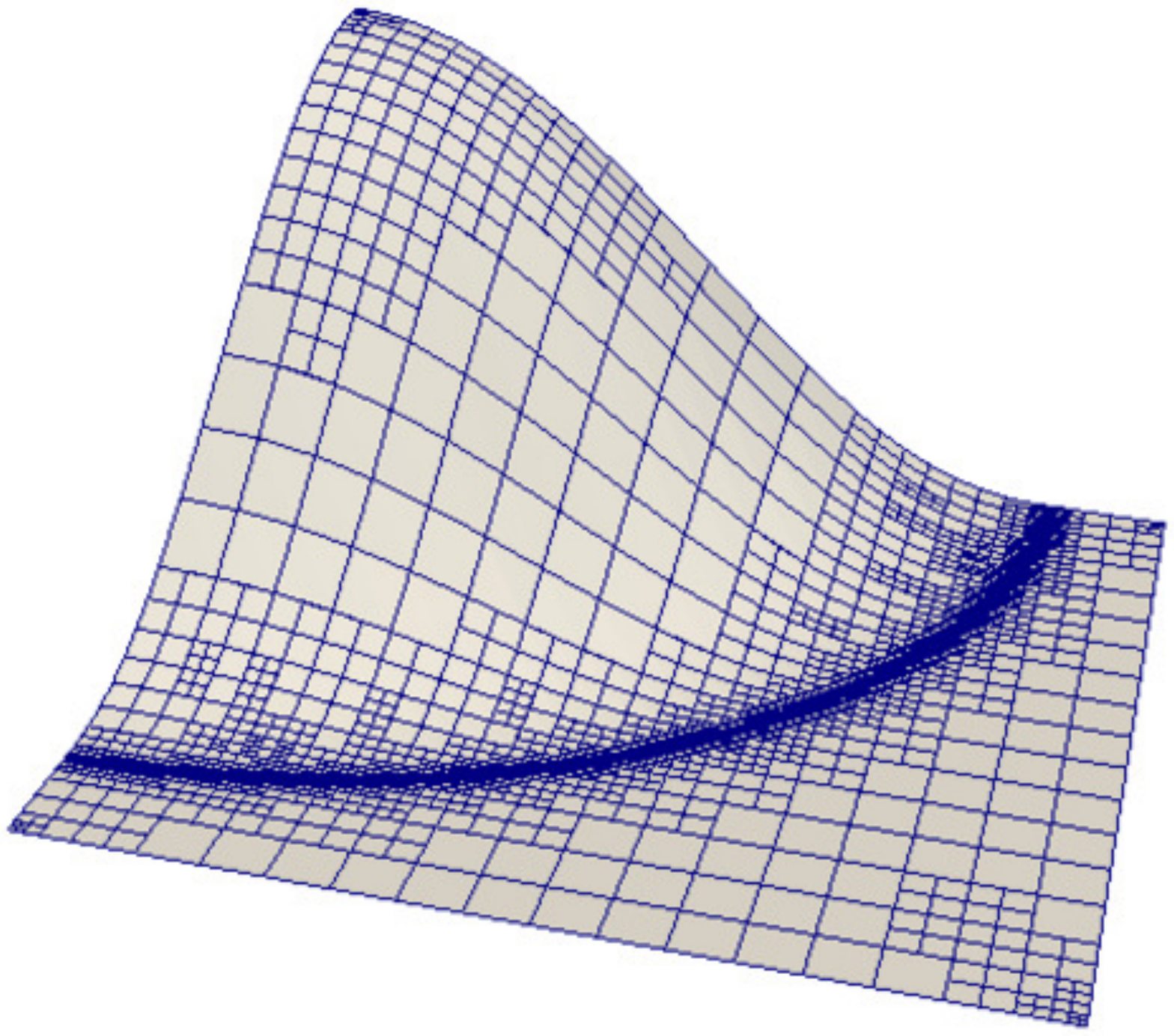} &\\
\end{tabular}
\end{center}
\caption{Smooth surface: Comparison between the algorithm based on the geometric estimators $\lambda$ (BCMN) with $r=k=3$ and $\mu$ (BD) with $r=3$, $k=2$. 
The subdivisions constructed to achieve an error tolerance $\eps =  5\times10^{-7}$ are provided in the left column for (top) BCMMN and (bottom) BD. 
In this case $\gamma$ is $C^{2,\frac 2 5}$ and the geometry approximation influences the rate of convergence.
Again, the adaptive algorithm based on $\mu$ produces a geometric refinement drastically less pronounced in order to achieve the targeted tolerance.
The right plot provide the evolution of the error estimators during the adaptive algorithms. 
With a one degree lower for the polynomial approximation of the geometry, the adaptive algorithm based on the BD estimator exhibits an optimal error decay $\textrm{DOF}^{-3/2}$, while the BCMMN algorithm appears to exhibit a slightly suboptimal error decay. For comparison, we also provide the outcome of a sequence of uniform refinement with $r=k=3$ leading to an error decay of $\textrm{DOF}^{-1}$.}\label{fig:C2alpha}
\end{figure}

\section{Perspectives} 
\label{s:perspectives}
In this section we briefly discuss two further questions raised by our work.  The first is the question of convergence of adaptive FEM naturally generated from our a posteriori estimators, as were employed in the numerical tests above.  The recent paper \cite{BCMMN16} from which we drew important elements of our technical structure proved convergence and optimality of an AFEM for the Laplace-Beltrami operator in which the geometric contribution is measured only by $\lambda$.  We hope to prove similar results for our AFEM employing $\mu$.

Another fundamental question concerns the regularity of $\gamma$.  Our results above strongly use the assumption that $\gamma$ is globally $C^2$, while \cite{BCMMN16} requires substantially less surface regularity.  If $\gamma$ is $C^2$, our results yield a geometric error of heuristic a priori order $h^2$.  On the other hand, if $\gamma$ is only $C^{1,1}$, our results have not been proven to apply, while the framework of \cite{BCMMN16} is still valid.  That work however gives a geometric error contribution of heuristic a priori order $h$.  Thus there is a jump in provable geometric error order from $h$ to $h^2$ when moving from $C^{1,1}$ to $C^2$ surfaces.  To our knowledge it is a completely open question whether this is an artifact of proof, and if so, how to provide a unified theory of a priori and a posteriori error estimation for surfaces of varying regularities.  

\bibliographystyle{siam}
\bibliography{lb_ho}

\end{document}